\documentclass[reqno, a4paper, 10pt]{amsart}

\usepackage[margin=1in]{geometry}
\usepackage[linesnumbered, noline, noend, ruled]{algorithm2e}
\usepackage{amsfonts}
\usepackage{amsmath}
\usepackage{amssymb}
\usepackage{amsthm}
\usepackage[english]{babel}
\usepackage{bbm}
\usepackage{booktabs}
\usepackage{braket}
\usepackage{float}
\usepackage[T1]{fontenc}
\usepackage{graphicx}
\usepackage{hyperref}
\usepackage[utf8]{inputenc}
\usepackage{lmodern}
\usepackage{nicefrac}
\usepackage[retain-explicit-plus, retain-zero-exponent, per-mode=symbol, output-exponent-marker=\text{e}]{siunitx}
\usepackage{subcaption}
\usepackage[table]{xcolor}
\usepackage{mathtools}
\usepackage{enumitem}
\usepackage{cases}
\usepackage[foot]{amsaddr}
\usepackage{cite}
\usepackage{tikz}

		\newtheorem{Theorem}{Theorem}[section]
		
		\newtheorem{Definition}[Theorem]{Definition}
		\newtheorem{Remark}[Theorem]{Remark}
		\newtheorem*{Remark*}{Remark}
		\newtheorem{Lemma}[Theorem]{Lemma}
		\newtheorem{Assumption}[Theorem]{Assumption}
		
		\setenumerate{label=(\arabic*), ref=(\arabic*)}
		
		\newcommand{\tr}[1]{\text{tr} #1 }

		\newcommand{\R}{\mathbb{R}}
		
		\newcommand{\norm}[2]{\left\lvert \left\lvert #1 \right\rvert \right\rvert_{#2}}
		\newcommand{\abs}[1]{\left\lvert #1 \right\rvert}
		\newcommand{\dx}[1]{\ \text{d}#1}

		\newcommand{\embedding}{\hookrightarrow}

		\makeatletter
		\DeclareOldFontCommand{\rm}{\normalfont\rmfamily}{\mathrm}
		\DeclareOldFontCommand{\sf}{\normalfont\sffamily}{\mathsf}
		\DeclareOldFontCommand{\tt}{\normalfont\ttfamily}{\mathtt}
		\DeclareOldFontCommand{\bf}{\normalfont\bfseries}{\mathbf}
		\DeclareOldFontCommand{\it}{\normalfont\itshape}{\mathit}
		\DeclareOldFontCommand{\sl}{\normalfont\slshape}{\@nomath\sl}
		\DeclareOldFontCommand{\sc}{\normalfont\scshape}{\@nomath\sc}
		\makeatother
		
		\numberwithin{equation}{section}

		\newcommand{\subin}{^\text{in}}
		\newcommand{\subwall}{^\text{wall}}
		\newcommand{\subout}{^\text{out}}
		\newcommand{\subdomain}{^\text{sub}}
		\newcommand{\subdes}{^\text{des}}
		
		\newcommand{\viscosity}{\mu}
		\newcommand{\density}{\rho}
		\newcommand{\hcapacity}{C_p}
		\newcommand{\conductivity}{\kappa}
		\newcommand{\htc}{\alpha}

		\newcommand{\normal}{n}
		
		\newcommand{\velocity}{u}
		\newcommand{\pressure}{p}
		\newcommand{\temperature}{T}
		\newcommand{\solution}{U}
		
		\newcommand{\strongvelo}{\velocity}
		\newcommand{\strongpres}{\pressure}
		\newcommand{\strongtemp}{\temperature}
		
		\newcommand{\testvelo}{\hat{v}}
		\newcommand{\testpres}{\hat{q}}
		\newcommand{\testtemp}{\hat{S}}
		\newcommand{\testsolution}{\hat{V}}
		
		\newcommand{\advelo}{v}
		\newcommand{\adpres}{q}
		\newcommand{\adtemp}{S}
		\newcommand{\adsolution}{P}
		
		\newcommand{\testadvelo}{\hat{u}}
		\newcommand{\testadpres}{\hat{p}}
		\newcommand{\testadtemp}{\hat{T}}
		\newcommand{\testadsolution}{\hat{U}}

		\newcommand{\fspace}{V}
		\newcommand{\pspace}{P}
		\newcommand{\tspace}{W}
		\newcommand{\statespace}{\mathcal{U}}

		\newcommand{\laplace}{\Delta}
		\newcommand{\grad}{\nabla}
		\newcommand{\divergence}[1]{\text{div}\left(#1\right)}
		\newcommand{\integral}[1]{\int_{#1}}
		
		\newcommand{\tdiv}[1]{\text{div}_\Gamma \left( #1 \right)}

		\newcommand{\lagrangian}{\mathcal{L}}
		\newcommand{\shapelagrangian}{\mathcal{G}}
		
		\newcommand{\admissiblegeom}{\mathcal{A}}
		\newcommand{\admissibledefo}{\Xi}
		
		\newcommand{\costfunction}{J}
		\newcommand{\reducedcostfunction}{\hat{J}}

		\newcommand{\vectorfield}{\mathcal{V}}
		
		\newcommand{\weightreg}{\lambda_3}
		\newcommand{\weightvelo}{\lambda_2}
		\newcommand{\weighttemp}{\lambda_1}
		\newcommand{\flow}{\Phi_t}

		\newcommand{\fluxdes}{Q\subdes}
		\newcommand{\flux}{Q}
		
		\newcommand{\transposed}{^\top}
			
\begin{document}

{\footnotesize
\begin{center}
	This is a preprint. The final version of this article can be found at \url{https://doi.org/10.1016/j.jmaa.2020.124476}.
\end{center}
}

\title{Shape Sensitivity Analysis for a Microchannel Cooling System}
\author{Sebastian Blauth$^{*,1,2}$}
\address{$^*$ Corresponding Author}
\address{$^1$ Fraunhofer ITWM, Kaiserslautern, Germany}
\email{\href{mailto:sebastian.blauth@itwm.fraunhofer.de}{sebastian.blauth@itwm.fraunhofer.de}}
\author{Christian Leith\"auser$^1$}
\email{\href{mailto:christian.leithaeuser@itwm.fraunhofer.de}{christian.leithaeuser@itwm.fraunhofer.de}}
\author{Ren\'e Pinnau$^2$}
\email{\href{pinnau@mathematik.uni-kl.de}{pinnau@mathematik.uni-kl.de}}
\address{$^2$ TU Kaiserslautern, Kaiserslautern, Germany}

\begin{abstract}
	We analyze the theoretical framework of a shape optimization problem for a microchannel cooling system. To this end, a cost functional based on the tracking of absorbed energy by the cooler as well as some desired flow on a subdomain of the cooling system is introduced. The flow and temperature of the coolant are modeled by a Stokes system coupled to a convection diffusion equation. We prove the well-posedness of this model on a domain transformed by the speed method. Further, we rigorously prove that the cost functional of our optimization problem is shape differentiable and calculate its shape derivative by means of a recent material derivative free adjoint approach.
	
	\medskip
	\noindent \textsc{Keywords. } Shape Optimization, Adjoint Approach, Optimization with PDE constraints, Microchannels
	
	\medskip
	\noindent \textsc{AMS subject classifications. } 49Q10, 49Q12, 35Q35, 76D55
\end{abstract}

\maketitle

\vspace{-0.5cm}
\section{Introduction}
\label{sec:introduction}

Since microchannel cooling systems have excellent thermal properties, they are widely used to cool devices that emit a lot of heat over a small area, e.g., electronic equipment or chemical microreactors (see, e.g., \cite{tuckerman, review_channels, khan2011review, naqiuddin2018overview} and the references therein). In particular, such cooling systems have high heat transfer coefficients due to the large specific surface area of the microchannels, and their low system inertia allows for a safe and reliable operation. 

In our earlier work \cite{blauth}, we introduced a model hierarchy for a shape optimization problem for such a microchannel cooling system and used this to investigate the optimization problem numerically. Here, we are going to analyze rigorously the theoretical aspects of the shape optimization problem in \cite{blauth}.

In the literature the shape design of microchannel geometries received a lot of attention, e.g., in \cite{juniper, sigmund, meyer, husain, afzal, kubo_topology, chen2016novel, chentopo, pan2008optimal, andreasen}. There, the shapes of the geometries were parametrized to reduce the optimization problems to finite dimensional ones. This approach obviously limits the shapes under consideration to the ones representable by the chosen parametrization. A more general, infinite dimensional approach is given by shape sensitivity analysis based on shape calculus. This technique investigates the sensitivity of a shape functional, i.e., a real-valued function defined on a subset of the power set of $\mathbb{R}^d$, with respect to infinitesimal variations of the domain (see, e.g., \cite{sokolowski_zolesio, delfour_zolesio}). In recent years, this has become a well-established approach used for various applications, e.g., for the shape optimization of electromagnetic devices \cite{gangl_shape, gangl2, semmler, leugering}, aircrafts \cite{gauger, kroll, hazra, schmidt_navier}, spin packs \cite{hohmann, leith, leith2, leith3}, automobiles \cite{itakura, magoulas, giannakoglou}, or pipes \cite{henrot, hohmann2019gradientbased}, and is also used to solve inverse problems, e.g., in tomography \cite{hintermueller, laurain, beretta}. To the best of our knowledge, the shape optimization of a microchannel geometry based on shape calculus has not been analytically investigated so far. However, note that the shape sensitivity analysis for similar state equations has already been considered in the literature, e.g., in \cite{lindemann, hohmann, hou, feppon, badra}.

To model the cooling system, we use a Stokes system coupled to a convection-diffusion equation, which describe the coolant's flow and temperature, respectively. We present a shape optimization problem for the cooler given by the minimization of a cost functional consisting of two tracking-type terms and a perimeter regularization. The former are based on the absorption of energy by the cooling system and on a desired velocity on a subdomain the cooler respectively. 

The main objective of this paper is to rigorously prove that the cost functional of the optimization problem mentioned above is shape differentiable and to calculate its shape derivative. For this, we use the recent material derivative free adjoint approach of \cite{sturm}. In particular, we transform a reference domain by the speed method, analyze the state system on the resulting transformed domain, and show its well-posedness for small transformations. Using this, we verify the assumptions of the modified Correa-Seeger theorem of \cite{sturm}, which we apply to our problem to calculate the shape derivative of our cost functional.

This paper is structured as follows. In Section~\ref{sec:problem_formulation} we present our mathematical model of the cooling system and the corresponding shape optimization problem, which we originally introduced in \cite{blauth}. We recall some basic results from shape calculus and state the modified version of the theorem of Correa and Seeger from \cite{sturm} in Section~\ref{sec:shape_calculus}. The analysis of the state system on a domain transformed by the speed method is carried out in Section~\ref{sec:analysis_state_system}. This is used in Section~\ref{sec:shape_differentiability}, where we transfer our setting to the one of the material derivative free adjoint approach of \cite{sturm} and prove the shape differentiability of our cost functional.

\section{Problem Formulation}
\label{sec:problem_formulation}

Before we introduce our mathematical model for the cooling system, we introduce some notations used throughout this paper. Afterwards, we formulate a PDE constrained shape optimization problem used to design the cooler's geometry. 


\subsection{Basic Notations}

Let $d = 2,3$ be the spatial dimension. We denote by $\Omega \subset \R^d$ an open, bounded, and connected Lipschitz domain that represents the geometry of the cooling system. Its boundary $\Gamma = \partial \Omega$ is divided into $\Gamma\subin$, $\Gamma\subout$, and $\Gamma\subwall$, corresponding to in- and outlet as well as wall boundary of the cooler. We assume that all three parts have a positive Lebesgue measure. Moreover, we denote by $\Omega\subdomain$ a subdomain of $\Omega$ consisting of the disjoint union of open and connected subsets of $\Omega$. A two-dimensional slice of such a microchannel geometry is shown in Figure~\ref{figure:full_geometry}.
\begin{figure}[b]
	\centering
	\begin{subfigure}[b]{0.49\textwidth}
		\includegraphics[width=\textwidth, trim= 8cm 0cm 0cm 0cm, clip]{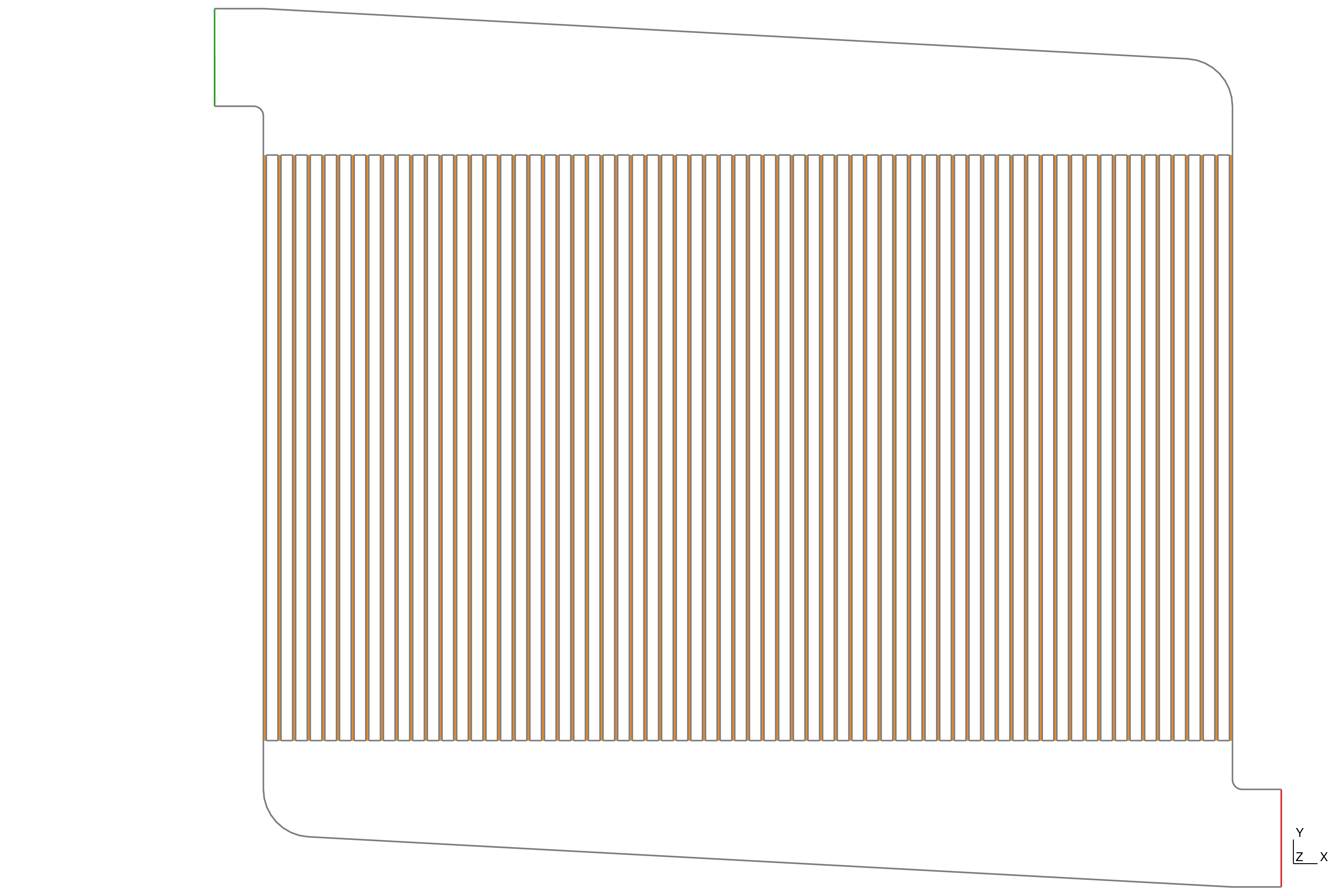}
		\caption*{The complete geometry.}
	\end{subfigure}
	\hfil
	\begin{subfigure}[b]{0.49\textwidth}
		\frame{\includegraphics[width=\textwidth, trim= 1cm 0cm 0.5cm 0cm, clip]{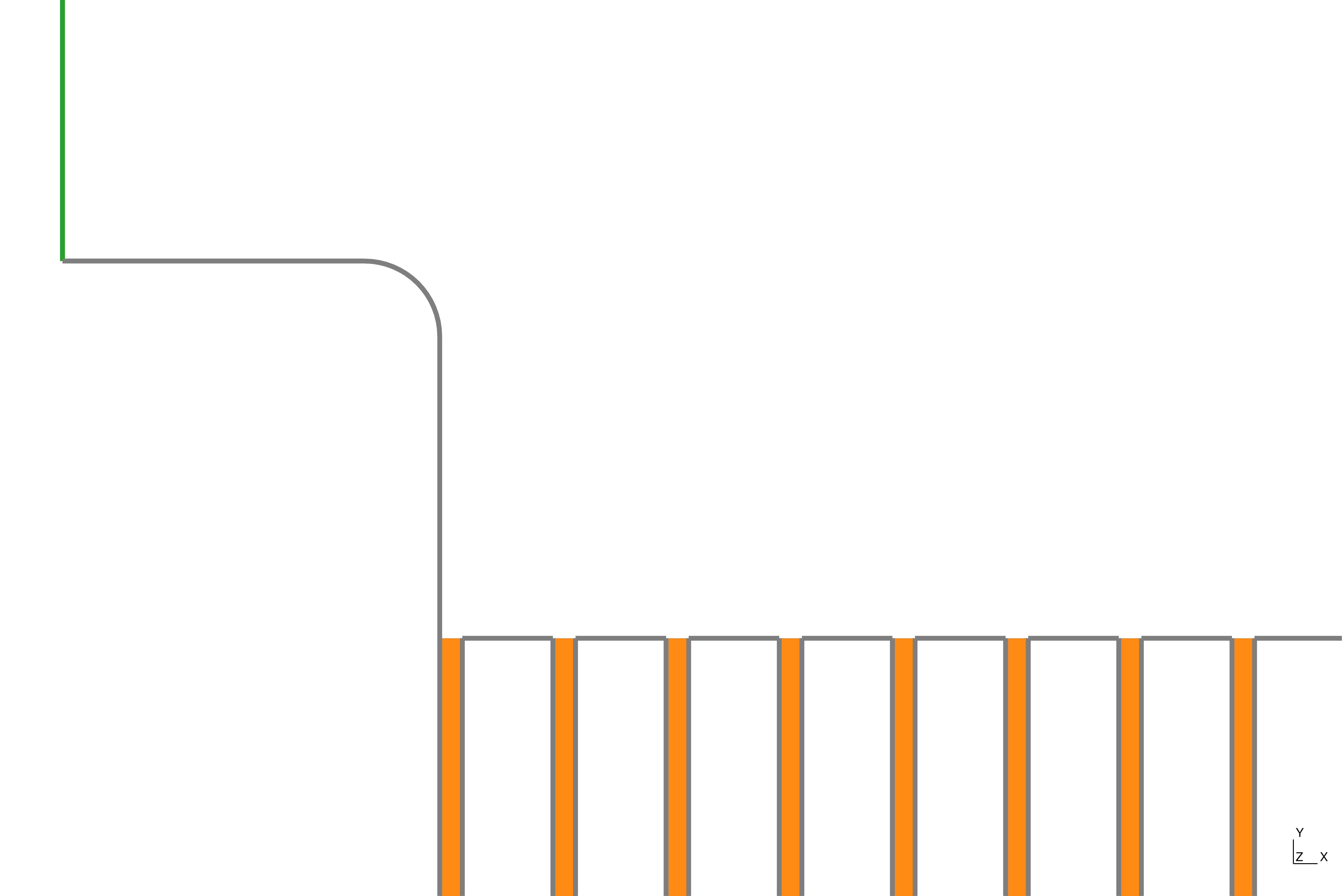}}
		\caption*{Zoom, showing inlet and some microchannels.}
	\end{subfigure}
	\caption{Two-dimensional slice of the domain $\Omega$, as used in \cite{blauth}, with inlet (green, top left), wall boundary (gray), and outlet (red, bottom right). The subdomain $\Omega\subdomain$ (orange) corresponds to the region of the microchannels.}
	\label{figure:full_geometry}
\end{figure}

For $k\in \mathbb{N}_{>0}$ the space $C^k_0(\R^d;\R^d)$ is defined as
\begin{equation*}
	C^k_0(\R^d;\R^d) := \Set{f \in C^k(\R^d;\R^d) | \text{supp}(f) \text{ is compact}},
\end{equation*}
i.e., the space of all $k$-times continuously differentiable mappings from $\R^d$ to $\R^d$ having compact support, endowed with the canonical norm (see, e.g., \cite{adams, alt}). Further, we introduce the Sobolev spaces
\begin{equation*}
	\begin{aligned}
		\fspace(\Omega) &:= \set{\velocity \in H^1(\Omega)^d | \velocity=0 \text{ on } \Gamma\subin \cup \Gamma\subwall}, \qquad \pspace(\Omega) := L^2(\Omega), \\
		\tspace(\Omega) &:= \set{\temperature \in H^1(\Omega) | \temperature=0 \text{ on } \Gamma\subin}, \qquad \statespace(\Omega) := \fspace(\Omega) \times \pspace(\Omega) \times \tspace(\Omega),
	\end{aligned}
\end{equation*}
equipped with the usual norms (see, e.g., \cite{alt, ern_guermond}). For generic elements $\solution, \adsolution \in \statespace(\Omega)$ we write $\solution = (\velocity, \pressure, \temperature)$ and $\adsolution = (\advelo, \adpres, \adtemp)$. For a Fr\'echet differentiable function $f \colon E \to F; x\mapsto f(x)$, where $E$ and $F$ are Banach spaces, we write $\partial_x f(x^*)[y]$ for the application of the Fr\'echet derivative of $f$ at $x^* \in E$ to some $y\in E$.
We call a real-valued function defined on a subset of $\set{\Omega | \Omega \subset \R^d}$ a shape functional. Throughout this paper, we denote by $C$ a generic positive constant, possibly taking different values in every place it appears, and write $C = C(\Omega)$ if the constant depends on the domain $\Omega$. We write $\lim_{t\searrow 0}$ for the one sided limit where $t$ approaches $0$ from the right and define $f'(0^+) = \lim_{t\searrow 0} \nicefrac{1}{t}\left( f(t) - f(0) \right)$.
We denote by $\text{id} \colon \R^d \to \R^d$ the identity mapping and by $I$ the identity matrix. The Jacobian of a vector-valued function $v$ is denoted by $Dv$ and its normal derivative is given by $\partial_\normal v = Dv\ \normal$, where $\normal$ is the outward unit normal on $\Gamma$. The tangential divergence of $v$ is given by $\tdiv{v} = \divergence{v} - (Dv\ n) n$ and the symmetric part of the Jacobian of $v$ is denoted by $\varepsilon(v) = \nicefrac{1}{2}\left( Dv + Dv\transposed \right)$.

\subsection{Mathematical Model}

We consider only the steady state of the cooling system and use the following Stokes system to model the behavior of the coolant
\begin{equation}
	\label{eq:stokes}
	\left\lbrace \quad
	\begin{alignedat}{2}
		-\viscosity \laplace \strongvelo + \grad \strongpres &= 0 \quad &&\text{ in } \Omega,\\
		\divergence{\strongvelo} &= 0 \quad &&\text{ in } \Omega,\\
		\strongvelo &= \velocity\subin \quad &&\text{ on } \Gamma\subin,\\
		\strongvelo &= 0 \quad &&\text{ on } \Gamma\subwall,\\
		\viscosity\ \partial_\normal \strongvelo - \strongpres \normal &= 0 \quad &&\text{ on } \Gamma\subout,
	\end{alignedat}
	\right.
\end{equation}
where $\strongvelo$ is the fluid's velocity, $\strongpres$ is its pressure, and $\viscosity$ denotes its viscosity. We have an inflow condition on $\Gamma\subin$ with inflow velocity $\velocity\subin$, which is chosen to satisfy $\velocity\subin \cdot \normal \leq 0$ a.e. on $\Gamma\subin$. Without loss of generality, we assume that $\velocity\subin \in H^1(\Omega)^d$  and $\velocity\subin = 0$ on $\Gamma\subwall$. Additionally, we use the usual no-slip boundary condition on $\Gamma\subwall$ and a do-nothing condition on $\Gamma\subout$, where the latter models the flow of the coolant out of $\Omega$ (see, e.g., \cite{john}).

To model the coolant's temperature distribution we use the following convection-diffusion equation
\begin{equation}
	\label{eq:convection-diffusion}
	\left\lbrace \quad
	\begin{alignedat}{2}
		- \grad \cdot \left( \conductivity \grad \strongtemp \right) + \density \hcapacity\ \strongvelo \cdot \grad \strongtemp &= 0 \quad &&\text{ in } \Omega,\\
		\strongtemp &= \temperature\subin \quad &&\text{ on } \Gamma\subin,\\
		\conductivity\ \partial_\normal \strongtemp + \htc \left( \strongtemp - \temperature\subwall \right) &= 0 \quad &&\text{ on } \Gamma\subwall,\\
		\conductivity\ \partial_\normal \strongtemp &= 0 \quad &&\text{ on } \Gamma\subout,
	\end{alignedat}
	\right.
\end{equation}
where $\strongtemp$ denotes the temperature of the fluid and $\strongvelo$ its velocity, obtained from \eqref{eq:stokes}. Further, $\conductivity$ is the fluid's thermal conductivity, $\density$ its density, and $\hcapacity$ denotes its heat capacity. On $\Gamma\subin$ we have a Dirichlet condition with inflow temperature $\temperature\subin$. As before, we assume  $\temperature\subin \in H^1(\Omega)$. Further, we assume that a heat source, which we do not model explicitly, is coupled to the cooling system. The wall temperature arising from this is denoted by $\temperature\subwall \in H^2(\R^d)$ and assumed to be known. Note, that the regularity of $\temperature\subwall$ is needed for the analysis of the shape optimization problem later on. The heat transfer between the heat source and the coolant is then modeled by a Robin boundary condition on $\Gamma\subwall$, where $\htc$ is the corresponding heat transfer coefficient. Finally, we have a homogeneous Neumann condition on $\Gamma\subout$ that models the behavior of the coolant's temperature at the outlet. Throughout this paper we assume that the physical parameters $\viscosity$, $\conductivity$, $\density$, $\hcapacity$, and $\htc$ are positive constants.

For the purposes of PDE constrained optimization the strong formulation of \eqref{eq:stokes} and \eqref{eq:convection-diffusion} is too strict (cf. \cite{hinze_pinnau_ulbrich, troeltzsch}), hence, we use the corresponding weak formulation introduced in the following. 
We write $\strongvelo = \velocity^0 + \velocity\subin$, $\strongpres = \pressure^0$, and $\strongtemp = \temperature^0 + \temperature\subin$ with $\velocity^0 \in \fspace(\Omega)$, $\pressure^0 \in \pspace(\Omega)$, and $\temperature^0 \in \tspace(\Omega)$, i.e., we do a so-called lifting of the inhomogeneous Dirichlet boundary conditions (see, e.g., \cite{ern_guermond, evans}), and obtain the following weak form of \eqref{eq:stokes} and \eqref{eq:convection-diffusion}
\begin{equation}
	\label{eq:weak_state_reference}
	\left\lbrace\quad
	\begin{aligned}
		&\text{Find } \solution^0 = (\velocity^0, \pressure^0, \temperature^0) \in \statespace(\Omega) \text{ such that } \\
		&\qquad \integral{\Omega} \viscosity\ D(\velocity^0 + \velocity\subin) : D\testvelo - \pressure^0\ \divergence{\testvelo} - \testpres\ \divergence{\velocity^0 + \velocity\subin} + \conductivity \grad (\temperature^0 + \temperature\subin) \cdot \grad \testtemp \dx{x} \\
		&\qquad + \integral{\Omega} \density \hcapacity \left(\velocity^0 + \velocity\subin\right) \cdot \grad (\temperature^0 + \temperature\subin)\ \testtemp \dx{x} + \integral{\Gamma\subwall} \htc \left((\temperature^0 + \temperature\subin) - \temperature\subwall\right) \testtemp \dx{s} = 0 \\
		&\text{for all } \testsolution = (\testvelo, \testpres, \testtemp) \in \statespace(\Omega).
	\end{aligned}
	\right.
\end{equation}
A generalized version of this weak formulation is analyzed in Section~\ref{sec:analysis_state_system}, which will give the well-posedness of \eqref{eq:weak_state_reference} as a special case.

\subsection{The Optimization Problem}
\label{sec:optimization_problem}

We want to improve the quality of the cooling system by modifying its shape. However, not all changes of the geometry lead to physically meaningful results. In particular, we assume that the boundaries $\Gamma\subin$ and $\Gamma\subout$ are fixed and must not be changed. This is reasonable as the cooling system is connected to other devices, such as the coolant supply, via these boundaries. The previously stated geometrical constraints are incorporated to the optimization problem by defining the following set of admissible domains
\begin{equation*}
	\admissiblegeom := \Set{\tilde{\Omega} | \tilde{\Omega} \subset \R^d \text{ is a bounded Lipschitz domain with } \tilde{\Gamma}\subin = \Gamma\subin \text{ and } \tilde{\Gamma}\subout = \Gamma\subout},
\end{equation*}
and only considering domains that are in $\admissiblegeom$.

To optimize the cooling system, we consider the following cost functional
\begin{equation*}
	\costfunction(\Omega, \solution) = \weighttemp \left( \flux(0, \temperature) - \fluxdes \right)^2 + \weightvelo \integral{\Omega\subdomain} \abs{(\velocity + \velocity\subin) - \velocity\subdes}^2 \dx{x} + \weightreg \integral{\Gamma} 1 \dx{s},
\end{equation*}
where we write $\solution = (\velocity, \pressure, \temperature)$ and $\flux(0,\temperature)$ is defined as
\begin{equation}
	\label{eq:flux_temperature}
	\flux(0, \temperature) := \integral{\Gamma\subwall} \htc \left( \temperature\subwall - \left( \temperature + \temperature\subin \right) \right) \dx{s}.
\end{equation}
We remark that the notation $\flux(0, \temperature)$ is introduced to be compatible with the calculations done later on in Section~\ref{sec:relation_to_abstract}.
Further, we assume that the weights $\lambda_i$ are non-negative, $\fluxdes \in \R$ is a constant and $\velocity\subdes \in H^1(\R^d)^d$. This yields the following shape optimization problem
\begin{equation}
	\label{eq:opt_problem}
	\min_{\Omega \in \admissiblegeom,\ \solution \in \statespace(\Omega)} \costfunction(\Omega, \solution) \quad \text{ subject to } \eqref{eq:weak_state_reference}.
\end{equation}
As in \cite{blauth}, we observe that minimizing the cost functional leads to the following. The first term ensures that $\flux(0,\temperature)$, which corresponds to the energy absorbed by the cooler, approaches a desired amount $\fluxdes$. For the physical application this models that the heat-producing device the cooling system is coupled to should be kept at an optimal working temperature, which occurs if the cooling system absorbs the desired amount of energy $\fluxdes$. 
The second term causes the flow velocity to approximate the desired velocity $\velocity\subdes$ on the subdomain $\Omega\subdomain$ by minimizing the $L^2$ distance between $\velocity + \velocity\subin$ and $\velocity\subdes$. This can, e.g., be used to achieve a uniform flow distribution through the microchannels, which helps to avoid the occurrence of hot-spots and, thus, increases the quality of the cooler (cf. \cite{blauth}).
The final term of the cost functional, a so-called perimeter regularization, penalizes the increase of surface area, which helps to keep the geometries bounded and enhances the numerical stability of the problem. Moreover, the perimeter regularization can, under suitable assumptions, be useful in proving the existence of a minimizer, as discussed in, e.g., \cite{perimeter_regularization}. However, we note that the investigation of existence of an optimal shape for problem \eqref{eq:opt_problem} is beyond the scope of this paper.

We used this cost functional in \cite{blauth} to optimize such a microchannel cooling system numerically, where we chose $\Omega\subdomain$ as the geometry of the microchannels and $\velocity\subdes$ as Poiseuille flow corresponding to a uniform flow distribution through the microchannels. The resulting optimized geometry from \cite{blauth} obtained with the weights
\begin{equation*}
	\weighttemp = \frac{1}{\left( \integral{\Gamma\subwall_0} \htc \left( \temperature\subwall - \left( \temperature + \temperature\subin \right) \right) \dx{s} - \fluxdes \right)^2}, \quad \weightvelo = \frac{1}{\integral{\Omega\subdomain_0} \abs{(\velocity + \velocity\subin) - \velocity\subdes}^2 \dx{x}}, \quad \weightreg = \frac{\num{1e-2}}{\integral{\Gamma_0} 1 \dx{s}}, 
\end{equation*}
where $\Omega_0$ is the initial domain and $\Gamma_0$ denotes its boundary, is depicted in Figure~\ref{figure:optimized}. Note, that for this optimization, the entire wall boundary $\Gamma\subwall$ was variable. However, to avoid degeneration of the geometry, the boundary $\partial\Omega\subdomain$ is only allowed to vary along the vertical direction so that the microchannels remain straight and can only change their length, as discussed in \cite{blauth}. 
\begin{figure}[b]
	\centering
	\begin{tikzpicture}
	\node at (0,0) {\includegraphics[width=0.5\textwidth]{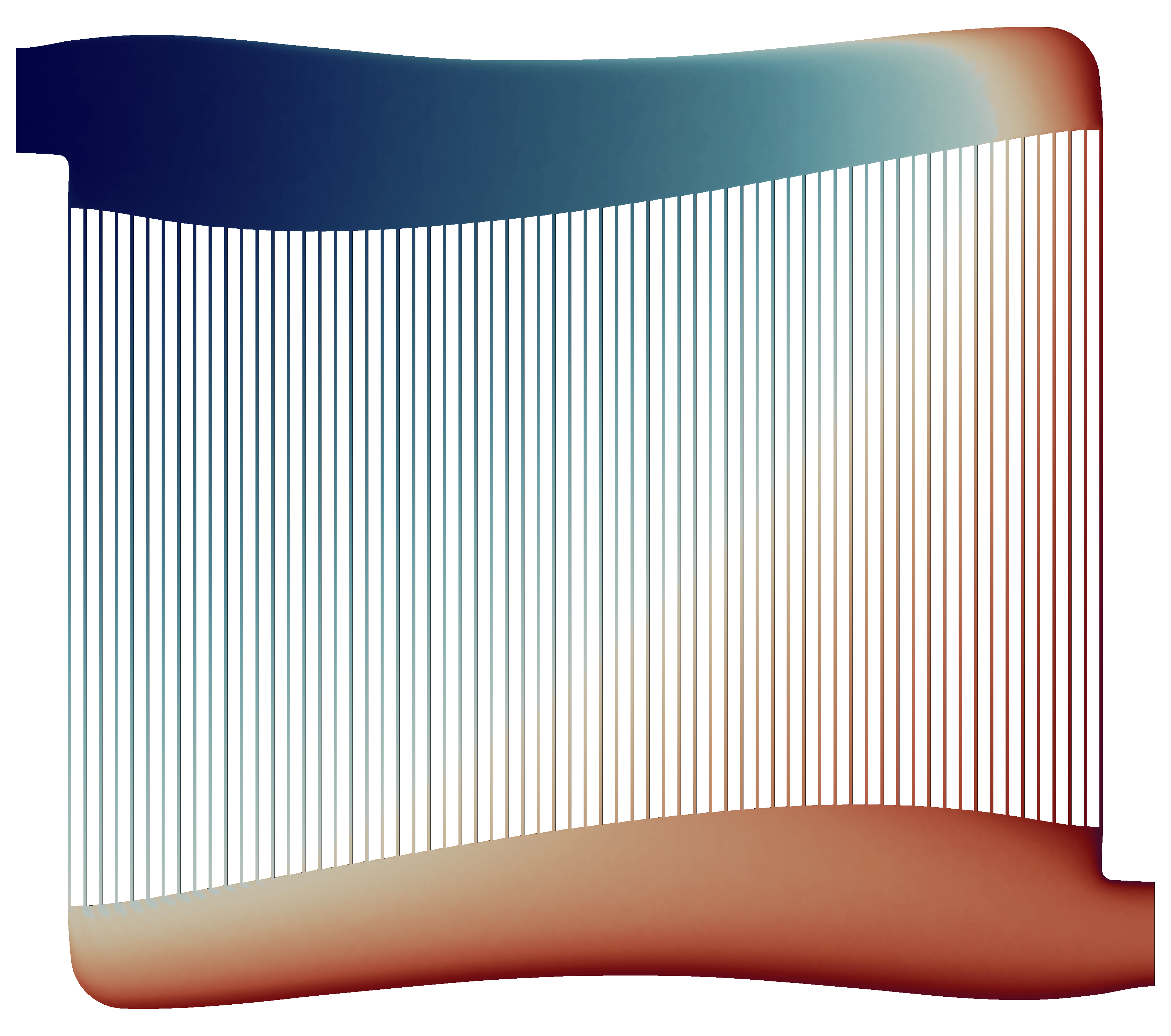}};
	\node at (5,0) {\includegraphics[width=0.1\textwidth, trim=15cm 85cm 30cm 30cm, clip]{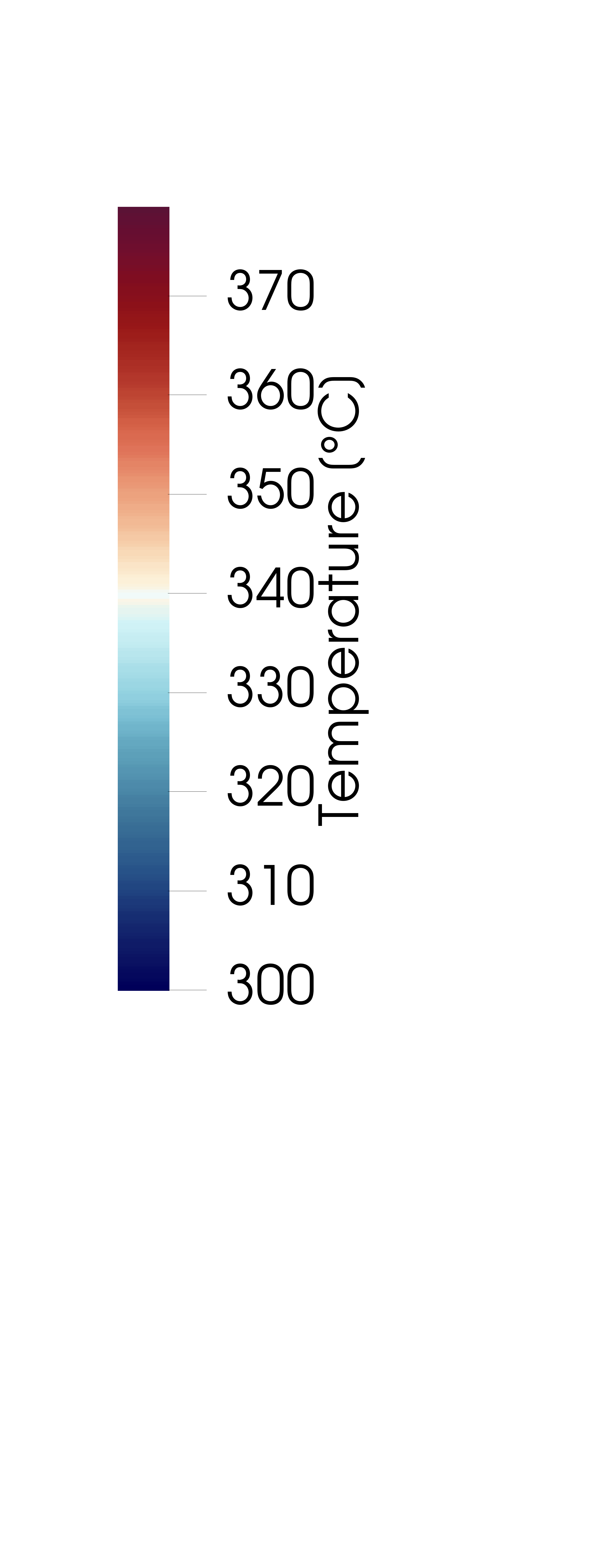}};
	\end{tikzpicture}
	\caption{Optimized geometry of the cooling system from \cite{blauth} showing the temperature of the coolant.}
	\label{figure:optimized}
\end{figure}
Finally, we remark that one could additionally consider the minimization of the system's pressure drop as additional goal for the optimization, which is often relevant for physical applications. This can be done, e.g., by adding the viscous dissipated energy $\integral{\Omega} D(\velocity + \velocity\subin) : D(\velocity + \velocity\subin) \dx{x}$ to the cost functional, which models the pressure power loss between in- and outlet, as discussed in \cite{hohmann}. We refer the reader, e.g., to \cite{hohmann, schmidt_navier} for a discussion of this term in the context of shape optimization. However, in this paper we do not further consider the pressure drop as it was of minor importance for the numerical results in \cite{blauth}.


In Section~\ref{sec:shape_differentiability} we show that the state system \eqref{eq:weak_state_reference} has a unique solution for every $\Omega \in \admissiblegeom$. Denoting by $\solution(\Omega)$ this weak solution on the domain $\Omega \in \admissiblegeom$, we introduce the reduced cost functional $\reducedcostfunction$ as
\begin{equation}
	\label{eq:reduced_cost_function}
	\reducedcostfunction(\Omega) := \costfunction(\Omega, \solution(\Omega)).
\end{equation}
It is easy to see that \eqref{eq:opt_problem} is equivalent to the following reduced shape optimization problem
\begin{equation}
	\label{eq:reduced_problem}
	\min_{\Omega \in \admissiblegeom} \reducedcostfunction(\Omega),
\end{equation}
where we have formally eliminated the PDE constraint. This problem is investigated in Section~\ref{sec:shape_differentiability}, after we give the necessary prerequisites for its analysis in Sections~\ref{sec:shape_calculus} and~\ref{sec:analysis_state_system}.

\section{Shape Calculus}
\label{sec:shape_calculus}

In this section, we discuss the theoretical foundation for analyzing shape optimization problems like \eqref{eq:reduced_problem}. We first introduce the notion of shape differentiability using the speed method and then present the material derivative free adjoint approach of \cite{sturm}, which we use in Section~\ref{sec:shape_differentiability} to calculate the shape derivative for our problem. A detailed introduction to shape sensitivity analysis can be found, e.g., in \cite{sokolowski_zolesio, delfour_zolesio}.

\subsection{Basic Results}

Throughout this section we consider a general bounded domain $\Omega \subset \R^d$, $d\in \mathbb{N}, d\geq 2$, with Lipschitz boundary $\Gamma = \partial \Omega$. The speed method considers the transformation of $\Omega$ under the flow of a vector field $\vectorfield$. In particular, let $\vectorfield \in C^k_0(\R^d; \R^d)$ for $k\geq 1$. The evolution of a point $x \in \Omega$ under the flow of $\vectorfield$ is defined as the solution of the following initial value problem
\begin{equation}
	\label{eq:ivp}
	\dot{x}(t) = \vectorfield(x(t)), \qquad x(0) = x.
\end{equation}
From classical ODE theory (see, e.g., \cite{teschl, deuflhard}) we know that \eqref{eq:ivp} has a unique solution $x(t)$ for all $t \in [0, \tau]$ if $\tau > 0$ is sufficiently small. The flow of $\vectorfield$ is then defined by
\begin{equation*}
	\flow^\vectorfield \colon \R^d \to \R^d;\quad x \mapsto \flow^\vectorfield(x) = x(t),
\end{equation*}
which is a diffeomorphism for all $t\in [0,\tau]$ (see, e.g., \cite{delfour_zolesio}). 
Throughout the rest of this paper, we drop the dependence on the vector field $\vectorfield$ and only write $\flow = \flow^\vectorfield$. Let us now introduce the shape derivative, following \cite[Chapter 9, Definition 3.4]{delfour_zolesio}.
\begin{Definition}
	\label{def:shape_derivative}
	Let $\tau > 0$ be sufficiently small, $\mathcal{S} \subset \set{\Omega | \Omega \subset \R^d}$, $J\colon \mathcal{S} \to \R$, and $\Omega \in \mathcal{S}$. Further, let $\vectorfield \in C^k_0(\R^d;\R^d)$ for $k\geq 1$, and $\flow$ be the flow associated to $\vectorfield$. Additionally, we assume that $\flow(\Omega) \in \mathcal{S}$ for all $t\in [0,\tau]$. We say that the shape functional $J$ has a Eulerian semi-derivative at $\Omega$ in direction $\vectorfield$ if the following limit exists
	\begin{equation*}
		dJ(\Omega)[\vectorfield] := \lim\limits_{t \searrow 0} \frac{J(\flow(\Omega)) - J(\Omega)}{t} = \left. \frac{d}{d t} J(\flow(\Omega)) \right\rvert_{t=0^+}.
	\end{equation*}
	
	Moreover, let $\admissibledefo$ be a topological vector subspace of $C^\infty_0(\R^d;\R^d)$. We call $J$ shape differentiable at $\Omega$ w.r.t. $\admissibledefo$ if it has an Eulerian semi-derivative at $\Omega$ in all directions $\vectorfield \in \admissibledefo$ and, additionally, the mapping
	\begin{equation*}
		\admissibledefo \to \R; \quad \vectorfield \mapsto dJ(\Omega)[\vectorfield]
	\end{equation*}
	is linear and continuous. In this case, we call $dJ(\Omega)[\vectorfield]$ the shape derivative of $J$ at $\Omega$ w.r.t. $\admissibledefo$ in direction $\vectorfield \in \admissibledefo$.
\end{Definition}
\begin{Remark*}
	The subspace $\admissibledefo$ in the previous definition is needed to incorporate the geometrical constraints of the optimization problem \eqref{eq:opt_problem}, namely that we minimize over all geometries $\Omega \in \admissiblegeom$, into the definition of the shape derivative (cf. Section~\ref{sec:analysis_state_system}).
\end{Remark*}

We repeat some basic results from shape calculus needed later on.
\begin{Lemma}
	\label{lem:shape_mappings}
	Let $\tau > 0$ be sufficiently small, $\vectorfield \in C^k_0(\R^d;\R^d)$ with $k\geq 1$, and $\flow$ be its associated flow.
		\begin{enumerate}
			\item The mapping $\xi \colon [0, \tau] \to C(\R^d; \R);\ t\mapsto \xi(t) := \det\left( D\flow \right)$ is continuously differentiable and we have
			\begin{equation*}
				\xi(0) = 1 \quad \text{ and } \quad \xi'(0^+) = \divergence{\vectorfield} \quad \text{ in } C(\R^d; \R).
			\end{equation*}

			\item The mapping $\omega \colon [0, \tau] \to C(\R^d; \R);\ t\mapsto \omega(t) := \xi(t) \abs{D\flow^{-\top} \normal}$ is continuously differentiable and it holds that
			\begin{equation*}
				\omega(0) = 1 \quad \text{ and } \quad \omega'(0^+) = \tdiv{\vectorfield} \quad \text{ in } C(\R^d; \R).
			\end{equation*}

			\item The mapping $A \colon [0, \tau] \to C(\R^d; \R^{d\times d});\ t\mapsto A(t) := \xi(t) D\flow^{-1} D\flow^{-\top}$ is continuously differentiable with
			\begin{equation*}
				A(0) = I \quad \text{ and } \quad A'(0^+) = \divergence{\vectorfield} I - 2\varepsilon(\vectorfield) \quad \text{ in } C(\R^d; \R^{d \times d}).
			\end{equation*}

			\item The mapping $B \colon [0, \tau] \to C(\R^d; \R^{d\times d});\ t\mapsto B(t) := \xi(t) D\flow^{-\top}$ is continuously differentiable and
			\begin{equation*}
				B(0) = I \quad \text{ and } \quad B'(0^+) = \divergence{\vectorfield} I - D\vectorfield\transposed \quad \text{ in } C(\R^d; \R^{d \times d}).
			\end{equation*}
		\end{enumerate}
\end{Lemma}
\begin{proof}
	The proof follows from results given in \cite[Chapters~2.15 and~2.18]{sokolowski_zolesio}.
\end{proof}

\begin{Lemma}
	\label{lem:shape_composition}
	Let $\tau > 0$ be sufficiently small, $t\in [0,\tau]$, $\vectorfield \in C^k_0(\R^d;\R^d)$ with $k\geq 1$, and $\flow$ be its associated flow. Further, we write $\Omega_t = \flow(\Omega)$ as well as $\Gamma_t = \flow(\Gamma)$. Then, we have:
	\begin{enumerate}
		\item Let $f\in L^1(\Omega)$ and $g\in L^1(\Gamma)$. Then, $f\circ \flow^{-1} \in L^1(\Omega_t)$, $g\circ \flow^{-1} \in L^1(\Gamma_t)$, and the following transformation formulas are valid
		\begin{equation*}
			\integral{\Omega_t} f \circ \flow^{-1} \dx{x} = \integral{\Omega} f\ \xi(t) \dx{x} \quad \text{ as well as } \quad \integral{\Gamma_t} g \circ \flow^{-1} \dx{s} = \integral{\Gamma} g\ \omega(t) \dx{s}.
		\end{equation*}

		\item Let $n \in \mathbb{N}_{>0}$, $1\leq p \leq \infty$, and $f \in W^{1,p}(\Omega)^n$. Then $f \circ \flow^{-1} \in W^{1,p}(\Omega_t)^n$ and we have the chain rule
		\begin{equation}
			\label{eq:chain_rule}
			D(f \circ \flow^{-1}) = (Df\ D\flow^{-1}) \circ \flow^{-1} \quad \text{ in } L^p(\Omega)^{n\times d}.
		\end{equation}
		In particular, for $n=1$ we get that $\grad(f \circ \flow^{-1}) = (D\flow^{-\top} \grad f ) \circ \flow^{-1}$.
		Moreover, for the divergence operator we can use \eqref{eq:chain_rule} and get for $v \in W^{1,1}(\Omega)^d$ that
		\begin{equation}
			\label{eq:transport_divergence}
			\divergence{v \circ \flow^{-1}} = \tr\left( D (v\circ \flow^{-1}) \right) = \tr\left( Dv\ D\flow^{-1} \right)\circ \flow^{-1} \quad \text{ in } L^1(\Omega).
		\end{equation}
		\label{item:chain_rule}

		\item For $f\in H^1(\R^d)$ it holds that
		\begin{equation*}
			\lim\limits_{t \searrow 0} \norm{f\circ \flow - f}{H^1(\R^d)} = 0 \quad \text{ and } \quad \lim\limits_{t \searrow 0} \norm{f\circ \flow^{-1} - f}{H^1(\R^d)} = 0.
		\end{equation*}

		\item Let $n\in \mathbb{N}_{>0}$, $1\leq p < \infty$, and $f\in W^{1,p}(\R^d)^n$. Then, the mapping $[0,\tau] \to L^p(\R^d)^n;\ t\mapsto f \circ \flow$ is differentiable and it holds that
		\begin{equation*}
			\left( f \circ \flow \right)'(0^+) = Df\ \vectorfield \quad \text{ in } L^p(\R^d)^n.
		\end{equation*}
		In particular, for $n=1$ we obtain $\left( f \circ \flow \right)'(0^+) = \grad f \cdot \vectorfield$ in $L^p(\R^d)$.
	\end{enumerate}
\end{Lemma}
\begin{proof}
	The proof follows from the following considerations:
	\begin{enumerate}
		\item The transformation formula can be found in, e.g., \cite{forster_3, peichl, sokolowski_zolesio}
		
		\item This is proved, e.g., in \cite{alt}.
		
		\item The third result follows easily from \ref{item:chain_rule} and \cite[Lemma~2.1, p. 527]{delfour_zolesio} and for \eqref{eq:transport_divergence} we refer to \cite{hohmann}.
		
		\item This follows from \cite[Chapter~2.25]{sokolowski_zolesio}. \qedhere
	\end{enumerate}
\end{proof}

\subsection{A Material Derivative Free Adjoint Approach}
\label{sec:sturm}

Now, we introduce the material derivative free adjoint approach of \cite{sturm} which we apply in Section~\ref{sec:shape_differentiability} to calculate the shape derivative of \eqref{eq:reduced_cost_function}. This approach modifies the classical theorem of Correa and Seeger (see, e.g., \cite{delfour_zolesio}) to obtain the differentiability of a Lagrangian with respect to a saddle point. We present this modified version of the theorem, which is due to \cite{sturm} and tailored to nonlinear shape optimization problems, in the following. Note, that there are several other approaches to prove the shape differentiability and to calculate the shape derivative, an overview of which can be found, e.g., in \cite{sturm_nonlinear}.

Let $E, F$ be Banach spaces and $\tau > 0$. We consider a mapping
\begin{equation*}
G \colon [0, \tau] \times E \times F \to \R; \quad \left(t, \solution, \adsolution\right) \mapsto G\left(t, \solution, \adsolution\right),
\end{equation*}
which we assume to be affine in the third component for all $(t, \solution)$ in $[0, \tau] \times E$. We define the set $E(t)$ as
\begin{equation*}
E(t) := \Set{\solution \in E | \partial_\adsolution G\left(t, \solution, 0\right)[\testsolution] = 0 \quad \text{ for all } \testsolution \in F}. 
\end{equation*}
To state the theorem, we need the following assumptions.
\begin{enumerate}[label=(A\arabic*), ref=(A\arabic*)]
	\item \label{ass:H0} We assume the following.
	\begin{enumerate}[label=(\roman*) ,ref=(\roman*)]
		\item \label{ass:H0i} The set $E(t)$ is single valued for all $t\in [0,\tau]$ and, hence, we write $E(t) = \set{\solution^t}$.
		
		\item \label{ass:H0ii} For all $t\in [0, \tau]$ and all $\adsolution\in F$ the mapping
		\begin{equation*}
		[0, 1] \to \R; \quad \theta \mapsto G\left(t, \theta \solution^t + (1-\theta)\solution^0, \adsolution\right) 
		\end{equation*}
		is absolutely continuous. 
		
		\item \label{ass:H0iii} For all $t\in [0,\tau]$, all $\testadsolution \in E$, and all $\adsolution\in F$ the mapping
		\begin{equation*}
		[0, 1] \to \R; \quad \theta \mapsto \partial_\solution G\left(t, \theta \solution^t + (1-\theta)\solution^0, \adsolution\right)[\testadsolution]
		\end{equation*}
		is well-defined and belongs to $L^1(0, 1)$.
	\end{enumerate}
\end{enumerate}
Thanks to Assumption~\ref{ass:H0}, we can define the solution set of the averaged adjoint system as
\begin{equation*}
	Y\left(t, \solution^t, \solution^0\right) := \Set{\adsolution \in F | \int_{0}^{1} \partial_\solution G\left(t, \theta \solution^t + (1-\theta)\solution^0, \adsolution\right)[\testadsolution] \dx{\theta} = 0 \quad \text{ for all } \testadsolution \in E }.
\end{equation*}
For $t=0$ we write $Y\left(0, \solution^0\right) := Y\left(0, \solution^0, \solution^0\right)$ which is the solution set of the usual adjoint system, i.e.,
\begin{equation*}
	Y\left(0, \solution^0\right) = \Set{\adsolution \in F | \partial_\solution G\left(t, \solution^0, \adsolution\right)[\testadsolution] = 0 \quad \text{ for all } \testadsolution \in E}.
\end{equation*}
Additionally, we have the following assumptions.
\begin{enumerate}[resume, label=(A\arabic*), ref=(A\arabic*)]
	\item \label{ass:H1} The derivative $\partial_t G\left(t, \solution^0, \adsolution\right)$ exists for all $t\in [0, \tau]$ and all $\adsolution \in F$.
	
	\item \label{ass:H2} The set $Y\left(t, \solution^t, \solution^0\right)$ is nonempty for all $t\in [0, \tau]$, and the set $Y\left(0, \solution^0\right)$ is single valued. Therefore, we write $Y(0, \solution^0) = \set{\adsolution^0}$. 
	
	\item \label{ass:H3} For every sequence $(t_n)_{n\in \mathbb{N}} \subset \R$ with $t_n \geq 0$ and $\lim\limits_{n\to \infty} t_n = 0$ there exists a subsequence $(t_{n_k})_{k\in \mathbb{N}}$ and $\adsolution^{t_{n_k}} \in Y\left(t_{n_k}, \solution^{t_{n_k}}, \solution^0\right)$ such that
	\begin{equation*}
	\lim_{\substack{k\to \infty \\ \theta\searrow 0}} \partial_t G\left(\theta, \solution^0, \adsolution^{t_{n_k}}\right) = \partial_t G\left(0, \solution^0, \adsolution^0\right).
	\end{equation*}
\end{enumerate}
After these preliminaries, we now state the modified version of the Correa-Seeger theorem.
\begin{Theorem}
	\label{thm:sturm}
	Let $G$ as before and let assumptions \ref{ass:H0}--\ref{ass:H3} be satisfied. Then, we have for any $\adsolution \in F$ that
	\begin{equation*}
		\left. \frac{d}{d t} G\left(t, \solution^t, \adsolution\right) \right\rvert_{t=0^+} = \partial_t G\left(0, \solution^0, \adsolution^0\right).
	\end{equation*}
\end{Theorem}
\begin{proof}
	The proof can be found in \cite[Theorem~3.1]{sturm}.
\end{proof}
In Section~\ref{sec:shape_differentiability} we establish the connection between this abstract setting and the one of Section~\ref{sec:optimization_problem}. This allows us to calculate the shape derivative by an application of the previous theorem, after verifying its assumptions in Section~\ref{sec:verification}.

\section{Analysis of the State System}
\label{sec:analysis_state_system}

In this section, we analyze the state system \eqref{eq:weak_state_reference} on a domain that is transformed by the speed method. In particular, we prove the existence and uniqueness of (weak) solutions as well as their continuous dependence on the data for small transformations. The results of this section are crucial for the shape sensitivity analysis carried out in Section~\ref{sec:shape_differentiability}. Throughout the rest of this paper, we restrict the spatial dimension $d \in \mathbb{N}$ to $2 \leq d \leq 4$, as all of the following results are valid for this choice. Note, that for the physical application, only the cases $d=2,3$ are relevant.

\subsection{Reformulation as an Abstract Problem}
\label{sec:reformulation}

Let $\Omega \in \admissiblegeom$ be as in Section~\ref{sec:problem_formulation}, i.e., $\Omega$ is a bounded domain with Lipschitz boundary $\Gamma$. We define the set of admissible vector fields by
\begin{equation}
	\label{eq:admissible_deformation}
	\admissibledefo := \Set{\vectorfield \in C^\infty_0(\R^d; \R^d) | \vectorfield = 0 \text{ on } \Gamma\subin \cup \Gamma\subout}.
\end{equation}
Obviously, $\admissibledefo$ is a vector subspace of $C^\infty_0(\R^d;\R^d)$, so that $\vectorfield \in \admissibledefo$ is admissible as direction for the shape derivative (cf. Definition~\ref{def:shape_derivative}). Throughout the rest of this section, let $\vectorfield \in \admissibledefo$ and $\flow$ be its associated flow. It is easy to see that we have $\flow = \text{id}$ on $\Gamma\subin \cup \Gamma\subout$. We assume that $\tau > 0$ is sufficiently small such that the transformed domain $\Omega_t = \flow(\Omega)$ is, like $\Omega$, a bounded Lipschitz domain for all $t\in [0,\tau]$ (cf. \cite{hofman}). In particular, this implies that $\Omega_t \in \admissiblegeom$ for all $t\in [0,\tau]$. The state system on $\Omega_t$ is given by
\begin{equation}
	\label{eq:weak_state}
	\left\lbrace \quad
	\begin{aligned}
		&\text{Find } \solution_t = (\velocity_t, \pressure_t, \temperature_t) \in \statespace(\Omega_t) \text{ such that } \\
		&\qquad \integral{\Omega_t} \viscosity\ D(\velocity_t + \velocity\subin_t) : D\testvelo - \pressure_t\ \divergence{\testvelo} - \testpres\ \divergence{\velocity_t + \velocity\subin_t} + \conductivity \grad (\temperature_t + \temperature\subin_t) \cdot \grad \testtemp \dx{x} \\
		&\qquad + \integral{\Omega_t} \density \hcapacity \left(\velocity_t + \velocity\subin_t\right) \cdot \grad (\temperature_t + \temperature\subin_t)\ \testtemp \dx{x} + \integral{\Gamma_t\subwall} \htc \left((\temperature_t + \temperature\subin_t) - \temperature\subwall\right) \testtemp \dx{s} = 0 \\
		&\text{for all } \testsolution = (\testvelo, \testpres, \testtemp) \in \statespace(\Omega_t),
	\end{aligned}
	\right.
\end{equation}
where we define $\velocity\subin_t := \velocity\subin \circ \flow^{-1}$ and $\temperature\subin_t := \temperature\subin \circ \flow^{-1}$. Due to the choice of $\vectorfield \in \admissibledefo$ we see that $\flow(\Gamma\subin) = \Gamma\subin$ and $\velocity\subin_t = \velocity\subin$ as well as $\temperature\subin_t = \temperature\subin$ on $\Gamma\subin$, i.e., these functions are valid liftings for the boundary conditions on $\Omega_t$. Furthermore, observe that we also have $\velocity\subin_t = 0$ on $\Gamma\subwall_t$.
From these considerations we see that the formulation \eqref{eq:weak_state} is equivalent to the original one from \eqref{eq:weak_state_reference} with the domain $\Omega$ replaced by $\Omega_t$. Using the change of variables $\solution_t = \solution^t \circ \flow^{-1}$ in addition to Lemmas~\ref{lem:shape_mappings} and~\ref{lem:shape_composition} reveals that \eqref{eq:weak_state} is equivalent to the following system
\begin{equation*}
	\left\lbrace\quad
	\begin{aligned}
		&\text{Find } \solution^t = (\velocity^t, \pressure^t, \temperature^t) \in \statespace(\Omega) \text{ such that } \\
		&\qquad \integral{\Omega} \viscosity \left(D(\velocity^t + \velocity\subin) A(t) \right) : D\testvelo  - \pressure^t\ \tr\left( D\testvelo\ B(t)\transposed \right) -\testpres\ \tr\left( D(\velocity^t + \velocity\subin) B(t)\transposed \right) \dx{x} \\
		&\qquad + \integral{\Omega}  \conductivity \left( A(t) \grad (\temperature^t + \temperature\subin ) \right) \cdot \grad \testtemp + \density \hcapacity \left(\velocity^t + \velocity\subin\right) \cdot \left( B(t) \grad (\temperature^t + \temperature\subin) \right) \ \testtemp \dx{x} \\
		&\qquad + \integral{\Gamma\subwall} \htc \left((\temperature^t + \temperature\subin) - \temperature\subwall \circ \flow\right) \testtemp\ \omega(t) \dx{s} = 0 \\
		&\text{for all } \testsolution = (\testvelo, \testpres, \testtemp) \in \statespace(\Omega).
	\end{aligned}
	\right. 
\end{equation*}
To recast this into an abstract setting, we introduce the following mappings
\begin{equation}
	\label{eq:def_bilinear_forms}
	\begin{aligned}
		&a \colon [0, \tau] \times \fspace(\Omega) \times \fspace(\Omega) \to \R; \quad (t, \velocity, \advelo) \mapsto \integral{\Omega} \viscosity \left( D \velocity\ A(t)\right) : D \advelo \dx{x}, \\
		&b \colon [0,\tau] \times \fspace(\Omega) \times \pspace(\Omega) \to \R; \quad (t, \velocity, \adpres) \mapsto -\integral{\Omega} \adpres\ \tr\left( D\velocity\ B(t)\transposed \right) \dx{x}, \\
		& c \colon [0, \tau] \times \fspace(\Omega) \times \tspace(\Omega) \times \tspace(\Omega) \to \R; \quad (t, \velocity, \temperature, \adtemp) \\
		&\quad \mapsto \integral{\Omega} \conductivity \left( A(t) \grad \temperature \right) \cdot \grad \adtemp + \density \hcapacity \left(\velocity + \velocity\subin\right) \cdot \left( B(t) \grad \temperature \right) \adtemp \dx{x} + \integral{\Gamma\subwall} \htc\ \temperature \adtemp\ \omega(t) \dx{s},
	\end{aligned}
\end{equation}
as well as
\begin{equation}
	\label{eq:def_linear_forms_state}
	\begin{aligned}
		&f_\velocity \colon [0, \tau] \times \fspace(\Omega) \to \R; \quad (t, \advelo) \mapsto -\integral{\Omega} \viscosity \left(D \velocity\subin\ A(t) \right) : D \advelo \dx{x}, \\
		&f_\pressure \colon [0, \tau] \times \pspace(\Omega) \to \R; \quad (t, \adpres) \mapsto \integral{\Omega} \adpres\ \tr\left( D\velocity\subin\ B(t)\transposed \right) \dx{x}, \\
		&f_\temperature \colon [0, \tau] \times \fspace(\Omega) \times \tspace(\Omega) \to \R; \quad (t, \velocity, \adtemp) \\
		&\quad \mapsto -\integral{\Omega} \conductivity \left( A(t) \grad \temperature\subin \right) \cdot \grad \adtemp + \density \hcapacity \left( \velocity + \velocity\subin \right)\cdot \left( B(t) \grad \temperature\subin \right)\ \adtemp \dx{x} \\
		&\quad\qquad - \integral{\Gamma\subwall} \htc \left( \temperature\subin - \temperature\subwall \circ \flow \right) \adtemp\ \omega(t) \dx{s}.	
	\end{aligned}
\end{equation}
Moreover, we define the mapping $e \colon [0,\tau] \times \statespace(\Omega) \times \statespace(\Omega) \to \R$ by
\begin{equation*}
	e(t, \solution, \adsolution) = a(t, \velocity, \testvelo) + b(t, \testvelo, \pressure) + b(t, \velocity, \testpres) + c(t, \velocity, \temperature, \testtemp) - f_\velocity(t ,\testvelo) - f_\pressure(t, \testpres) - f_\temperature(t, \velocity, \testtemp),
\end{equation*}
where $\solution = (\velocity, \pressure, \temperature) \in \statespace(\Omega)$ and $\adsolution = (\testvelo, \testpres, \testtemp) \in \statespace(\Omega)$. Finally, using the definitions of the mappings in \eqref{eq:def_bilinear_forms} and \eqref{eq:def_linear_forms_state} reveals that \eqref{eq:weak_state} is equivalent to
\begin{equation}
	\label{eq:abstract_weak_state}
	\text{Find } \solution^t = \left(\velocity^t, \pressure^t, \temperature^t \right) \in \statespace(\Omega) \text{ such that } \qquad e\left(t, \solution^t, \testsolution\right) = 0 \qquad \text{ for all } \testsolution = \left(\testvelo, \testpres, \testtemp \right) \in \statespace(\Omega).
\end{equation}
It is easy to see that we can rewrite \eqref{eq:abstract_weak_state} equivalently as the following one-way coupled system
\begin{equation}
	\label{eq:one_way_stokes}
	\left \lbrace \quad
	\begin{aligned}
		&\text{Find } (\velocity^t, \pressure^t) \in \fspace(\Omega) \times \pspace(\Omega) \text{ such that } \\
		&\qquad 
		\begin{alignedat}{2}
			a(t, \velocity^t, \testvelo) + b(t, \testvelo, \pressure^t) &= f_\velocity(t, \testvelo) \quad &&\text{ for all } \testvelo \in \fspace(\Omega), \\
			b(t, \velocity^t, \testpres) &= f_\pressure(t, \testpres) \quad &&\text{ for all } \testpres \in \pspace(\Omega), 
		\end{alignedat}
	\end{aligned}
	\right.
\end{equation}
and
\begin{equation}
	\label{eq:one_way_temperature}
	\left\lbrace\quad
	\begin{aligned}
		&\text{Find } \temperature^t \in \tspace(\Omega) \text{ such that } \\
		&\qquad c(t, \velocity^t, \temperature^t, \testtemp) = f_\temperature(t, \velocity^t, \testtemp) \quad \text{ for } (\velocity^t, \pressure^t) \text{ solving } \eqref{eq:one_way_stokes} \text{ and all } \testtemp \in \tspace(\Omega).
	\end{aligned}
	\right.
\end{equation}
We use this equivalent formulation in the following section to prove the well-posedness of the state system.

\subsection{Well-Posedness of the Abstract Problem}
To show the well-posedness of the abstract problem \eqref{eq:abstract_weak_state} we first analyze the mappings from \eqref{eq:def_bilinear_forms} and \eqref{eq:def_linear_forms_state}.
We start by introducing the following assumption.
\begin{Assumption}
	\label{ass:smallness}
	We assume that the inlet velocity $\velocity\subin \in H^1(\Omega)^d$ is sufficiently small, i.e., $\norm{\velocity\subin}{H^1(\Omega)^d}$ is sufficiently small.
\end{Assumption}
\begin{Remark*}
	This assumption is only needed to ensure the coercivity of the bilinear form to prove the well-posedness of the convection-diffusion equation \eqref{eq:one_way_temperature} with a velocity obtained from the Stokes system \eqref{eq:one_way_stokes}. The well-posedness of the latter is completely independent of the assumption (cf. the proof of Lemma~\ref{lem:analysis_state_system}). 
	
	Note, that while Assumption~\ref{ass:smallness} is not restrictive mathematically, it might be physically. An alternative assumption that would ensure the coercivity for the convection-diffusion equation is given by $(\velocity + \velocity\subin) \cdot \normal \geq 0$ a.e. on $\Gamma\subout$, i.e., the fluid should not re-enter the domain at the outlet. Physically, this assumption is very reasonable and we observed that it is indeed satisfied for the numerical optimization in \cite{blauth}. However, this assumption is very restrictive mathematically, as it has to hold for all admissible domains $\Omega$, which is why we use Assumption~\ref{ass:smallness} instead. Another way to ensure the coercivity would be a modification of the boundary conditions for the temperature, e.g., by prescribing the outlet velocity so that we have no re-entrant flow, but this would not fit to our physical model.

\end{Remark*}
\begin{Lemma}
	\label{lem:operator_analysis}
	Let $\tau > 0$ be sufficiently small and $t\in [0,\tau]$, and let Assumption~\ref{ass:smallness} hold. Then, there exist constants $C(\Omega) > 0$, independent of $t$, such that we have the following:
	\begin{enumerate}
		\item The mapping $a(t, \cdot, \cdot)$ is a symmetric, continuous, and coercive bilinear form. In particular, it holds
		\begin{equation*}
			a(t, \velocity, \velocity) \geq C(\Omega) \norm{\velocity}{\fspace(\Omega)}^2 \quad \text{ for all } \velocity \in \fspace(\Omega).
		\end{equation*}

		\item The mapping $b(t, \cdot, \cdot)$ is bilinear, continuous, and satisfies the following LBB condition
		\begin{equation*}
			\sup_{\velocity \in \fspace(\Omega), \velocity\neq 0} \frac{b(t, \velocity, \adpres)}{\norm{\velocity}{\fspace(\Omega)} } \geq C(\Omega) \norm{\adpres}{\pspace(\Omega)} \quad \text{ for all } \adpres\in \pspace(\Omega),
		\end{equation*}
		which is named after Ladyzhenskaya, Babu\v{s}ka, and Brezzi, and is also known as inf-sup condition (cf. \cite{john}).

		\item Let $w \in \fspace(\Omega)$ with $\norm{w}{\fspace(\Omega)}$ sufficiently small. Then, the mapping $c(t, w, \cdot, \cdot)$ is a continuous and coercive bilinear form which satisfies
		\begin{equation*}
			c(t, w, \temperature, \temperature) \geq C(\Omega) \norm{\temperature}{\tspace(\Omega)}^2 \quad \text{ for all } \temperature \in \tspace(\Omega).
		\end{equation*}
	\end{enumerate}
\end{Lemma}
\begin{proof}
	Let $t\in [0,\tau]$ and $w \in \fspace(\Omega)$ with $\norm{w}{\fspace(\Omega)}$ sufficiently small. It is obvious from their definitions that the mappings $a(t, \cdot, \cdot)$, $b(t, \cdot, \cdot)$ and $c(t, w, \cdot, \cdot)$ are bilinear.
	\begin{enumerate}
		\item Let $\velocity, \advelo \in \fspace(\Omega)$. The symmetry of $a(t, \cdot, \cdot)$ directly follows from the symmetry of $A(t)$. Due to Lemma~\ref{lem:shape_mappings} the mapping $t\mapsto \norm{A(t)}{L^\infty(\Omega)^{d\times d}}$ is continuous and, thus, attains its maximum on the interval $[0,\tau]$. Using this in conjunction with the H\"older inequality yields the continuity of $a(t, \cdot, \cdot)$ via
		\begin{equation*}
			\begin{aligned}
				\abs{\integral{\Omega} \viscosity \left( D \velocity\ A(t)\right) : D \advelo \dx{x}} &\leq C \sup_{t \in [0,\tau]} \norm{A(t)}{L^\infty(\Omega)^{d\times d}} \norm{\velocity}{\fspace(\Omega)} \norm{\advelo}{\fspace(\Omega)} \\
				&\leq C \norm{\velocity}{\fspace(\Omega)} \norm{\advelo}{\fspace(\Omega)},
			\end{aligned}
		\end{equation*}
		where the maximum (and, therefore, also the constant $C>0$) is independent of $t$.
		We obtain the desired coercivity by the following calculation
		\begin{equation*}
			\begin{aligned}
				a(t, \velocity, \velocity) & = \integral{\Omega} \viscosity\ D\velocity : D\velocity \dx{x} - \integral{\Omega} \viscosity \left( D\velocity (I - A(t)) \right) : D\velocity \dx{x} \\
				&\geq \left( C(\Omega) - C \sup_{t \in [0,\tau]} \norm{I - A(t)}{L^\infty(\Omega)^{d\times d}} \right) \norm{\velocity}{\fspace(\Omega)}^2,
			\end{aligned}
		\end{equation*}
		where we used the Poincar\'e and H\"older inequality for the first and second term, respectively. As a consequence of Lemma~\ref{lem:shape_mappings} it is easy to see that the mapping $t\mapsto \norm{I - A(t)}{L^\infty(\Omega)^{d\times d}}$ is continuous and, thus, attains its maximum on $[0,\tau]$. This maximum is independent of $t$ and, since $A(0) = I$, it becomes arbitrarily small if we choose $\tau > 0$ sufficiently small. Therefore, we finally obtain a constant $C(\Omega)$ independent of $t$ such that
		\begin{equation*}
			a(t, \velocity, \velocity) \geq C(\Omega) \norm{\velocity}{\fspace(\Omega)}^2 \quad \text{ for all } \velocity \in \fspace(\Omega).
		\end{equation*}

		\item Let $\adpres \in \pspace(\Omega)$ and $\velocity \in \fspace(\Omega)$. The continuity of the mapping $b(t, \cdot, \cdot)$ follows from the H\"older inequality and Lemma~\ref{lem:shape_mappings} due to
		\begin{equation*}
			\begin{aligned}
				\abs{\integral{\Omega} \adpres\ \tr\left( D\velocity B(t)\transposed \right) \dx{x}} &\leq C \sup_{t \in [0,\tau]} \norm{B(t)\transposed}{L^\infty(\Omega)^{d\times d}} \norm{q}{\pspace(\Omega)} \norm{\velocity}{\fspace(\Omega)} \\
				&\leq C \norm{q}{\pspace(\Omega)} \norm{\velocity}{\fspace(\Omega)}.
			\end{aligned}
		\end{equation*}
		For the LBB condition we first observe that the divergence operator $\text{div}\colon \fspace(\Omega) \to \pspace(\Omega)$ is surjective by a slight modification of the proof given in \cite[Lemma~4.9]{ern_guermond}. Then, applying \cite[Lemma~2.1]{gatica}, which is a consequence of the open mapping theorem, yields the existence of a constant $C(\Omega) > 0$ such that
		\begin{equation}
			\label{eq:lbb}
			\sup_{\velocity \in \fspace(\Omega), \velocity \neq 0} \frac{\integral{\Omega} q\ \divergence{u} \dx{x} }{\norm{\velocity}{\fspace(\Omega)}} \geq C(\Omega) \norm{\adpres}{\pspace(\Omega)} \quad \text{ for all } \adpres \in \pspace(\Omega).
		\end{equation}
		Similarly to, e.g., \cite{fumagalli, girault_lbb}, we have the following for the LBB condition
		\begin{equation*}
			\begin{aligned}
				\sup_{\velocity\in \fspace(\Omega), \velocity\neq 0} \frac{b(t, \velocity, \adpres)}{\norm{\velocity}{\fspace(\Omega)}} &= \sup_{\velocity \in \fspace(\Omega), \velocity\neq 0} \left( \frac{\integral{\Omega} \adpres\ \divergence{\velocity} \dx{x}}{\norm{\velocity}{\fspace(\Omega)}}  - \frac{\integral{\Omega} \adpres\ \tr\left( D\velocity (I - B(t)\transposed) \right) \dx{x}}{\norm{\velocity}{\fspace(\Omega)}} \right)  \\
				%
				&\geq \left( C(\Omega) - C \sup_{t \in [0,\tau]} \norm{I - B(t)\transposed}{L^\infty(\Omega)^{d\times d}} \right) \norm{\adpres}{\pspace(\Omega)} \\
				&\geq C(\Omega) \norm{\adpres}{\pspace(\Omega)},
			\end{aligned}
		\end{equation*}
		where we used \eqref{eq:lbb} to estimate the first term and H\"older's inequality for the second. As before, the final estimate and the fact that $C(\Omega)$ is independent of $t$ follows from Lemma~\ref{lem:shape_mappings}.

		\item Let $w$ be given as above and $\temperature, \adtemp \in \tspace(\Omega)$. For the continuity of $c(t, w, \cdot, \cdot)$ we use the continuous embedding $H^1(\Omega) \embedding L^4(\Omega)$ (due to the Sobolev embedding theorem), the H\"older inequality, and Lemma~\ref{lem:shape_mappings} to obtain
		\begin{equation*}
			\begin{aligned}
				&\abs{\integral{\Omega} \conductivity \left( A(t) \grad \temperature \right) \cdot \grad \adtemp + \density \hcapacity \left(w + \velocity\subin\right) \cdot \left( B(t) \grad \temperature \right)\ \adtemp \dx{x} + \integral{\Gamma\subwall} \htc\ \temperature \adtemp\ \omega(t) \dx{s}} \\
				%
				%
				\leq\ & C(\Omega) \sup_{t \in [0,\tau]} \left( \norm{A(t)}{L^\infty(\Omega)^{d\times d}} + \norm{B(t)}{L^\infty(\Omega)^{d\times d}} + \norm{\omega(t)}{L^\infty(\Omega)}  \right) \norm{w}{\fspace(\Omega)} \norm{\temperature}{\tspace(\Omega)} \norm{\adtemp}{\tspace(\Omega)} \\
				\leq\ & C(\Omega) \norm{w}{\fspace(\Omega)} \norm{\temperature}{\tspace(\Omega)} \norm{\adtemp}{\tspace(\Omega)}.
			\end{aligned}
		\end{equation*}
		For the coercivity of $c(t, w, \cdot, \cdot)$ we observe that
		\begin{equation*}
			\begin{aligned}
				c(t, w, \temperature, \temperature) &= \integral{\Omega} \conductivity \grad \temperature \cdot \grad \temperature \dx{x} - \integral{\Omega} \conductivity \left( (I - A(t)) \grad \temperature \right) \cdot \grad \temperature \dx{x}  \\
				&\quad + \integral{\Omega} \density\hcapacity \left( w + \velocity\subin \right) \cdot \left( B(t) \grad \temperature \right) \temperature \dx{x} + \integral{\Gamma\subwall} \htc\ \temperature^2\ \omega(t) \dx{s} \\
				&\geq \left( C(\Omega) - C \sup_{t \in [0,\tau]} \norm{I - A(t)}{L^\infty(\Omega)^{d\times d}} \right) \norm{\temperature}{\tspace(\Omega)}^2 \\
				&\quad - C \sup_{t \in [0,\tau]} \norm{B(t)}{L^\infty(\Omega)^{d\times d}} \left( \norm{w}{\fspace(\Omega)} + \norm{\velocity\subin}{H^1(\Omega)^d} \right) \norm{\temperature}{\tspace(\Omega)}^2 \\
				&\geq C(\Omega) \norm{\temperature}{\tspace(\Omega)}^2,
			\end{aligned}
		\end{equation*}
		where we used the Poincar\'e inequality for the first term, the H\"older inequality for the second and third term, and the fact that $\omega(t) > 0$ for sufficiently small $\tau$ for the final term. The last estimate is obtained analogously from Lemma~\ref{lem:shape_mappings}, where we also used the smallness of $w$ and $\velocity\subin$. \qedhere
	\end{enumerate}
\end{proof}

\begin{Lemma}
	\label{lem:analysis_linear_forms}
	Let $\tau > 0$ be sufficiently small, $t\in [0,\tau]$ and $\velocity \in \fspace(\Omega)$. Then, the mappings $f_\velocity(t, \cdot)$, $f_\pressure(t, \cdot)$, and $f_\temperature(t, \velocity, \cdot)$ are linear and continuous. In particular, there exists a constant $C>0$ that is independent of $t$ such that
	\begin{equation*}
		\begin{aligned}
			&\abs{f_\velocity(t, \advelo)} \leq C \norm{\velocity\subin}{H^1(\Omega)^d} \norm{\advelo}{\fspace(\Omega)} \quad \text{ for all } \advelo \in \fspace(\Omega), \\
			&\abs{f_\pressure(t, \adpres)} \leq C \norm{\velocity\subin}{H^1(\Omega)} \norm{\adpres}{\pspace(\Omega)} \quad \text{ for all } \adpres \in \pspace(\Omega), \\
			&\abs{f_\temperature(t, \velocity, \adtemp)} \leq C \left( \left( \norm{\velocity}{\fspace(\Omega)} + \norm{\velocity\subin}{H^1(\Omega)^d} + 1 \right) \norm{\temperature\subin}{H^1(\Omega)} + \norm{\temperature\subwall}{H^2(\R^d)} \right) \norm{\adtemp}{\tspace(\Omega)} \\
			&\hspace{6em} \quad \text{ for all } \velocity \in \fspace(\Omega) \text{ and  all } \adtemp \in \tspace(\Omega).
		\end{aligned}
	\end{equation*}
\end{Lemma}
\begin{proof}
	The proof follows from the H\"older inequality and Lemma~\ref{lem:shape_mappings} applied analogously to the proof of Lemma~\ref{lem:operator_analysis}.
\end{proof}

\begin{Lemma}
	\label{lem:analysis_state_system}
	Let Assumption~\ref{ass:smallness} hold and $\tau > 0$ be sufficiently small. Then, for every $t\in [0, \tau]$ the state system \eqref{eq:abstract_weak_state} is well-posed, i.e., there exists a unique solution $\solution^t = (\velocity^t, \pressure^t, \temperature^t) \in \statespace(\Omega)$ of \eqref{eq:abstract_weak_state} that depends continuously on the data. In particular, there is a constant $C(\Omega) > 0$ which is independent of $t$ such that
	\begin{equation}
		\label{eq:energy_estimates}
		\begin{aligned}
			&\norm{\solution^t}{\statespace(\Omega)} = \norm{\velocity^t}{\fspace(\Omega)} + \norm{\pressure^t}{\pspace(\Omega)} + \norm{\temperature^t}{\tspace(\Omega)} \\
			\leq\ &C(\Omega) \left( \norm{\velocity\subin}{H^1(\Omega)^d} + \norm{\temperature\subin}{H^1(\Omega)} + \norm{\temperature\subwall}{H^2(\R^d)} \right) \norm{\adtemp}{\tspace(\Omega)} \quad \text{ for all } t\in [0, \tau].
		\end{aligned}
	\end{equation}
\end{Lemma}
\begin{proof}
	Let $t\in [0,\tau]$. For the proof we switch to the equivalent one-way coupled formulation of \eqref{eq:abstract_weak_state} given by \eqref{eq:one_way_stokes} and \eqref{eq:one_way_temperature}. 
	
	We start by investigating the Stokes system \eqref{eq:one_way_stokes}. In Lemma~\ref{lem:operator_analysis} we have shown that the bilinear form $a(t, \cdot, \cdot)$ is continuous and coercive and that the bilinear form $b(t, \cdot, \cdot)$ is continuous and satisfies the LBB condition \eqref{eq:lbb}. Thanks to Lemma~\ref{lem:analysis_linear_forms} we know that the right-hand sides $f_\velocity(t, \cdot)$ and $f_\pressure(t, \cdot)$ are continuous linear functionals. Therefore, the Lions-Lax-Milgram theorem (see, e.g., \cite[Theorem~2.34]{ern_guermond} or \cite{john, girault}) gives us the existence of a unique solution $(\velocity^t, \pressure^t) \in \fspace(\Omega) \times \pspace(\Omega)$ of \eqref{eq:one_way_stokes}. Moreover, there exists a constant $C(\Omega) > 0$ independent of $t$ such that
	\begin{equation}
		\label{eq:help_const_velo}
		\norm{\velocity^t}{\fspace(\Omega)} + \norm{\pressure^t}{\pspace(\Omega)} \leq C(\Omega) \norm{\velocity\subin}{H^1(\Omega)^d} \quad \text{ for all } t\in [0,\tau].
	\end{equation}
	
	For the convection-diffusion equation \eqref{eq:one_way_temperature} we note that $\velocity^t$ is sufficiently small in the $\fspace(\Omega)$ norm due to \eqref{eq:help_const_velo} and Assumption~\ref{ass:smallness}, which ensures that $\norm{\velocity\subin}{H^1(\Omega)^d}$ is sufficiently small. As a consequence of Lemma~\ref{lem:operator_analysis}, $c(t, \velocity^t, \cdot, \cdot)$ is a continuous and coercive bilinear form. Further, $f_\temperature(t, \velocity^t, \cdot)$ is a continuous linear functional (cf. Lemma~\ref{lem:analysis_linear_forms}). Therefore, the Lax-Milgram lemma (see, e.g., \cite{ern_guermond, evans, girault}) yields the existence and uniqueness of $\temperature^t \in \tspace(\Omega)$ that solves \eqref{eq:one_way_temperature}. Due to \eqref{eq:help_const_velo}, we have a constant $C(\Omega) > 0$ independent of $t$ such that
	\begin{equation}
		\label{eq:help_const_temp}
		\norm{\temperature^t}{\tspace(\Omega)} \leq C(\Omega) \left( \left( \norm{\velocity\subin}{H^1(\Omega)^d} + 1 \right) \norm{\temperature\subin}{H^1(\Omega)} + \norm{\temperature\subwall}{H^2(\R^d)} \right) \norm{\adtemp}{\tspace(\Omega)} \quad \text{ for all } t\in [0,\tau].
	\end{equation}
	
	Note, that the constants in \eqref{eq:help_const_velo} and \eqref{eq:help_const_temp} are independent of $t$ as they are rational functions of the constants of Lemmas~\ref{lem:operator_analysis} and~\ref{lem:analysis_linear_forms}, which do not dependent on $t$. Finally, adding \eqref{eq:help_const_velo} and \eqref{eq:help_const_temp} completes the proof. \qedhere

\end{proof}

\section{Shape Differentiability of the Problem}
\label{sec:shape_differentiability}

In this section we prove the shape differentiability of the reduced cost functional $\reducedcostfunction$, defined in \eqref{eq:reduced_cost_function}. To do this, we first establish the connection between the abstract differentiability result of Section~\ref{sec:sturm} and our setting. Afterwards, we verify the assumptions of Theorem~\ref{thm:sturm} and calculate the shape derivative of \eqref{eq:reduced_cost_function}. 

Throughout this section, we assume that $\Omega\in \admissiblegeom$ is a fixed reference domain, $\vectorfield \in \admissibledefo$ (cf. \eqref{eq:admissible_deformation}), and that $\flow$ is the flow associated to $\vectorfield$. Further, we assume that Assumption~\ref{ass:smallness} is valid and for $t\in [0,\tau]$, with $\tau >0$ sufficiently small, we denote by $\solution^t = (\velocity^t, \pressure^t, \temperature^t)$ the unique solution of \eqref{eq:abstract_weak_state} (cf. Lemma~\ref{lem:analysis_state_system}). Finally, note that each constant $C$ or $C(\Omega)$ appearing in this section is independent of $t$.

\subsection{Relation to the Abstract Differentiability Result of Section~\ref{sec:sturm}}
\label{sec:relation_to_abstract}

We introduce the following Lagrangian associated to \eqref{eq:opt_problem} on $\Omega$
\begin{equation*}
	\lagrangian(\Omega, \solution, \adsolution) = \costfunction(\Omega, \solution) + e(0, \solution, \adsolution).
\end{equation*}
As $\solution^0 = \solution(\Omega)$ is the unique solution of \eqref{eq:abstract_weak_state} for $t=0$, we have that $e(0, \solution^0, \adsolution) = 0$ for all $\adsolution \in \statespace(\Omega)$ and, hence, it holds that
\begin{equation*}
	\lagrangian(\Omega, \solution^0, \adsolution) = \costfunction(\Omega, \solution^0) + e(0, \solution^0, \adsolution) = \costfunction(\Omega, \solution(\Omega)) = \reducedcostfunction(\Omega) \quad \text{ for all } \adsolution \in \statespace(\Omega).
\end{equation*}
Using the speed method we define $\Omega_t = \flow(\Omega)$ and recall that $\Omega_t \in \admissiblegeom$ (cf. Section~\ref{sec:analysis_state_system}). Then, the cost functional on the transformed domain $\Omega_t$ is given by
\begin{equation*}
	\costfunction(\Omega_t, \tilde{\solution}) = \weighttemp \left( \integral{\Gamma\subwall_t} \htc \left( \temperature\subwall - (\tilde{\temperature} + \temperature\subin_t) \right) \dx{s} - \fluxdes \right)^2 + \weightvelo \integral{\Omega\subdomain_t} \abs{(\tilde{\velocity} + \velocity\subin_t) - \velocity\subdes }^2 \dx{x} + \weighttemp \integral{\Gamma_t} 1 \dx{s},
\end{equation*}
where we write $\tilde{\solution} = (\tilde{\velocity}, \tilde{\pressure}, \tilde{\temperature}) \in \statespace(\Omega_t)$ as well as $\velocity\subin_t = \velocity\subin \circ \flow^{-1}$ and $\temperature\subin_t = \temperature\subin \circ \flow^{-1}$ as in \eqref{eq:weak_state}. Now, using the change of variables $\tilde{\solution} = \solution \circ \flow^{-1}$ as well as Lemmas~\ref{lem:shape_mappings} and~\ref{lem:shape_composition} allows us to pull-back the cost functional to the reference domain $\Omega$ and we obtain
\begin{equation*}
		\costfunction(\Omega_t, \solution \circ \flow^{-1}) = \weighttemp \left( \flux(t, \temperature)  - \fluxdes \right)^2 + \weightvelo \integral{\Omega\subdomain} \abs{(\velocity + \velocity\subin) - \velocity\subdes \circ\flow}^2 \xi(t) \dx{x} + \weightreg \integral{\Gamma} \omega(t) \dx{s},
\end{equation*}
where we write $\solution = (\velocity, \pressure, \temperature)$ and define
\begin{equation*}
	\flux(t, \temperature) := \integral{\Gamma\subwall} \htc \left( \temperature\subwall \circ \flow - (\temperature + \temperature\subin) \right) \omega(t) \dx{s},
\end{equation*}
which reduces to \eqref{eq:flux_temperature} for $t=0$ due to Lemma~\ref{lem:shape_mappings}. To get a consistent notation, we write $j(t, \solution) := \costfunction(\Omega_t, \solution \circ \flow^{-1})$. We define the shape Lagrangian $\shapelagrangian \colon [0, \tau] \times \statespace(\Omega) \times \statespace(\Omega) \to \R$ corresponding to \eqref{eq:opt_problem} by
\begin{equation*}
	\shapelagrangian(t, \solution, \adsolution) = \lagrangian(\Omega_t, \solution \circ \flow^{-1}, \adsolution \circ \flow^{-1}) = j(t, \solution) + e(t, \solution, \adsolution).
\end{equation*}
Note, that this parametrization allows us to work with the fixed function space $\statespace(\Omega)$ on the reference domain $\Omega$ instead of using the varying space $\statespace(\Omega_t)$, significantly simplifying the analysis (cf. \cite{delfour_zolesio}). 
For $\solution^t$, i.e., the solution of the state system \eqref{eq:abstract_weak_state}, we have that
\begin{equation*}
	e(t,\solution^t,\adsolution) = 0 \quad \text{ for all } \adsolution \in \statespace(\Omega).
\end{equation*}
Moreover, we define the function $\solution_t := \solution^t \circ \flow^{-1}$. As $\flow$ is a diffeomorphism, we can read the calculations at the beginning of Section~\ref{sec:reformulation} backwards and observe that $\solution_t = \solution(\Omega_t)$, i.e., $\solution_t$ is the solution of the state system \eqref{eq:weak_state}. This yields that
\begin{equation}
	\label{eq:reduced_saddle}
	\shapelagrangian(t, \solution^t, \adsolution) = j(t, \solution^t) + e(t, \solution^t, \adsolution) = \costfunction(\Omega_t, \solution_t) = \costfunction(\Omega_t, \solution(\Omega_t)) = \reducedcostfunction(\Omega_t) \quad \text{ for all } \adsolution \in \statespace(\Omega).
\end{equation}
Using \eqref{eq:reduced_saddle} we can now calculate the shape derivative in the following way
\begin{equation}
	\label{eq:relation}
	d\reducedcostfunction(\Omega)[\vectorfield] = \lim\limits_{t\searrow 0} \frac{\reducedcostfunction(\Omega_t) - \reducedcostfunction(\Omega)}{t} = \left. \frac{d}{d t} \reducedcostfunction(\Omega_t) \right\rvert_{t=0^+} = \left. \frac{d}{d t} \shapelagrangian(t, \solution^t, \adsolution) \right\rvert_{t=0^+} \quad \text{ for all } \adsolution \in \statespace(\Omega).
\end{equation}
To calculate the final derivative in the above equation, we apply Theorem~\ref{thm:sturm} to the shape Lagrangian $\shapelagrangian$. However, we first have to verify its assumptions, which we do in the following.

\subsection{Verification of the Conditions of Theorem~\ref{thm:sturm}}
\label{sec:verification}

We set $E = F = \statespace(\Omega)$ and use $G = \shapelagrangian$, where the latter is affine in the $\adsolution$ component by construction.

\subsection*{Verification of Assumption \ref{ass:H0}}

It is easy to see that $\partial_\adsolution \shapelagrangian(t, \solution^t, 0)[\testsolution] = e(t, \solution^t, \testsolution)$ for all $\testsolution \in \statespace(\Omega)$. Therefore, we see that the set $E(t)$ can be rewritten equivalently as
\begin{equation*}
	E(t) = \Set{\solution \in \statespace(\Omega) | e(t, \solution, \testsolution) = 0 \text{ for all } \testsolution \in \statespace(\Omega)},
\end{equation*}
i.e., $E(t)$ is the solution set of the state system \eqref{eq:abstract_weak_state}. Thanks to Lemma~\ref{lem:analysis_state_system} we know that \eqref{eq:abstract_weak_state} has a unique solution $\solution^t \in \statespace(\Omega)$ and, thus, $E(t) = \set{\solution^t}$ is single valued, showing part \ref{ass:H0i} of Assumption~\ref{ass:H0}. 

Furthermore, we note that the (pulled-back) cost functional $j(t,\solution)$ is Fr\'echet differentiable in $\solution$ as it consists of linear, continuous mappings in $\solution$ and compositions of these with continuously Fr\'echet differentiable functions. Additionally, the state system $e(t, \solution, \adsolution)$ is multilinear and continuous in $\solution$ (cf. Lemma~\ref{lem:operator_analysis}) and, hence, also Fr\'echet differentiable (see, e.g., \cite{troeltzsch, hinze_pinnau_ulbrich}). From this, we obtain that $\shapelagrangian$ is Fr\'echet differentiable w.r.t. $\solution$. As the mapping
\begin{equation*}
	[0,1] \to \statespace(\Omega); \quad \theta \mapsto \theta \solution^t + (1- \theta) \solution^0
\end{equation*}
is obviously Fr\'echet differentiable, we get by the chain rule (see, e.g., \cite{hinze_pinnau_ulbrich, troeltzsch}) that
\begin{equation*}
	[0, 1] \to \R; \quad \theta \mapsto \shapelagrangian\left(t, \theta \solution^t + (1-\theta)\solution^0, \hat{\adsolution}\right)
\end{equation*}
is also Fr\'echet differentiable, which implies part \ref{ass:H0ii} of Assumption~\ref{ass:H0}. 

Finally, it is easy to see that the mapping
\begin{equation*}
	[0, 1] \to \R; \quad \theta \mapsto \partial_\solution \shapelagrangian\left(t, \theta \solution^t + (1-\theta)\solution^0, \adsolution\right)[\testadsolution]
\end{equation*}
is indeed continuous and, thus, in $L^1(0,1)$, completing the verification of Assumption~\ref{ass:H0}.

\subsection*{Verification of Assumption \ref{ass:H1}}

Due to Lemmas~\ref{lem:shape_mappings} and~\ref{lem:shape_composition} we know that all terms of $\shapelagrangian$ involving $t$ are continuously differentiable w.r.t. $t$. A straightforward calculation shows that we have the following
\begin{equation}
	\label{eq:partial_time}
	\begin{aligned}
		&\partial_t \shapelagrangian(t, \solution^0, \adsolution) \\
		=\ &2\weighttemp\ \left( \flux(t, \temperature^0) - \fluxdes \right)\ \integral{\Gamma\subwall} \htc \left( \temperature\subwall \circ \flow - (\temperature^0 + \temperature\subin) \right) \omega'(t) + \htc \left(\left(\grad \temperature\subwall \cdot \vectorfield\right) \circ \flow\right) \omega(t) \dx{s} \\
		&+ \weightvelo \integral{\Omega\subdomain} \abs{(\velocity^0 + \velocity\subin) - \velocity\subdes \circ \flow}^2 \xi'(t) -2 \left( (\velocity^0 + \velocity\subin) - \velocity\subdes \circ \flow \right) \cdot \left( \left( D\velocity\subdes\ \vectorfield \right) \circ \flow\right) \xi(t) \dx{x} \\
		&+ \weightreg \integral{\Gamma} \omega'(t) \dx{s} + \integral{\Omega} \viscosity \left(D(\velocity^0 + \velocity\subin) A'(t) \right) : D\advelo  - \pressure^0\ \tr\left( D\advelo B'(t)\transposed \right) \dx{x} \\
		&+ \integral{\Omega} -\adpres\ \tr\left( D(\velocity^0 + \velocity\subin) B'(t)\transposed \right) + \conductivity \left( A'(t) \grad (\temperature^0 + \temperature\subin ) \right) \cdot \grad \adtemp \dx{x} \\
		&+ \integral{\Omega} \density \hcapacity \left(\velocity^0 + \velocity\subin\right) \cdot \left( B'(t) \grad (\temperature^0 + \temperature\subin) \right) \ \adtemp \dx{x} \\
		&+ \integral{\Gamma\subwall} \htc \left((\temperature^0 + \temperature\subin) - \temperature\subwall\circ \flow\right) \adtemp\ \omega'(t) - \htc \left(\left(\grad \temperature\subwall \cdot \vectorfield\right) \circ \flow\right) \adtemp \omega(t) \dx{s},
	\end{aligned}
\end{equation}
where $\adsolution = (\advelo, \adpres, \adtemp)$. Due to the H\"older inequality and the fact that $\temperature\subwall \in H^2(\R^d)$ (cf. Section~\ref{sec:problem_formulation}) this is well-defined. Hence, we have verified Assumption~\ref{ass:H1}.

\subsection*{Verification of Assumption \ref{ass:H2} and Adjoint System}

For this, we prove that the averaged adjoint system
\begin{equation}
	\label{eq:sturm_averaged_adjoint}
	\text{Find } \adsolution^t \in \statespace(\Omega) \text{ such that } \qquad \int_{0}^{1} \partial_\solution \shapelagrangian\left(t, \theta \solution^t + (1-\theta)\solution^0, \adsolution^t\right)[\testadsolution] \dx{\theta} = 0 \qquad \text{ for all } \testadsolution \in \statespace(\Omega)
\end{equation}
has a unique solution. Carrying out the integration over $\theta$ yields the following equivalent system
\begin{equation*}
	\left\lbrace \quad
	\begin{aligned}
		&\text{Find } \adsolution^t = (\advelo^t, \adpres^t, \adtemp^t) \in \statespace(\Omega) \text{ such that } \\
		&\quad \integral{\Omega} \conductivity \left(A(t) \grad \testadtemp \right) \cdot \grad \adtemp^t + \density\hcapacity \left( \nicefrac{1}{2} \left( \velocity^t + \velocity^0 \right) + \velocity\subin \right) \cdot \left( B(t) \grad \testadtemp \right) \adtemp^t \dx{x} + \integral{\Gamma\subwall} \htc\ \testadtemp \adtemp^t\ \omega(t) \dx{x} \\
		&\quad\ \ + \integral{\Omega} \viscosity \left( D\testadvelo\ A(t) \right) : D\advelo^t - \testadpres\ \tr\left( D\advelo^t\ B(t)\transposed \right) - \adpres^t\ \tr\left( D\testadvelo\ B(t)\transposed \right) \dx{x} \\
		&\quad\ \ + \integral{\Omega} \density \hcapacity \testadvelo \cdot \left( B(t) \grad \left( \nicefrac{1}{2}(\temperature^t + \temperature^0) + \temperature\subin \right) \right) \adtemp^t \dx{x} \\
		&\quad = 2\weighttemp \left( \flux(t, \nicefrac{1}{2}(\temperature^t + \temperature^0)) - \fluxdes \right) \integral{\Gamma\subwall} \htc\ \testadtemp\ \omega(t) \dx{s} \\
		&\quad\ \ - 2\weightvelo \integral{\Omega\subdomain} \left( \nicefrac{1}{2} (\velocity^t + \velocity^0) + \velocity\subin - \velocity\subdes\circ \flow \right) \cdot \testadvelo\ \xi(t) \dx{x} \\
		&\text{for all } \testadsolution = (\testadvelo, \testadpres, \testadtemp) \in \statespace(\Omega).
	\end{aligned}
	\right.
\end{equation*}
In the following, we reformulate this as an abstract problem and prove its well-posedness in analogy to Section~\ref{sec:analysis_state_system}. For this, we introduce the mappings
\begin{equation}
	\label{eq:averaged_linear_forms}
	\begin{aligned}
		&g_\temperature \colon [0, \tau] \times \tspace(\Omega) \to \R; \quad (t, \testadtemp) \mapsto g_\temperature(t, \testadtemp) = 2 \weighttemp \left( \flux(t, \nicefrac{1}{2}(\temperature^t + \temperature^0)) - \fluxdes \right) \integral{\Gamma\subwall} \hspace{-1em} \htc\ \testadtemp\ \omega(t) \dx{s}, \\
		&g_\velocity \colon [0, \tau] \times \tspace(\Omega) \times \fspace(\Omega) \to \R; \quad (t, \adtemp, \testadvelo) \\
		&\quad \mapsto -2\weightvelo \integral{\Omega\subdomain} \left(\nicefrac{1}{2} ( \velocity^t + \velocity^0 ) + \velocity\subin - \velocity\subdes \right) \cdot \testadvelo\ \xi(t) \dx{x} - \integral{\Omega} \density\hcapacity\ \testadvelo \cdot \left( B(t) \grad (\nicefrac{1}{2}(\temperature^t + \temperature^0) + \temperature\subin) \right) \adtemp \dx{x},
		%
	\end{aligned}
\end{equation}
where $\solution^t = (\velocity^t, \pressure^t, \temperature^t)$ solves \eqref{eq:abstract_weak_state}. We define the mapping $e^* \colon [0,\tau] \times \statespace(\Omega) \times \statespace(\Omega) \to \R$ by
\begin{equation*}
	e^*(t, \solution, \adsolution) := a(t, \testadvelo, \advelo) + b(t, \advelo, \testadpres) + b(t, \testadvelo, \adpres) + c\left(t, \nicefrac{1}{2} (\velocity^t + \velocity^0), \testadtemp, \adtemp\right) - g_\velocity(t, \adtemp, \testadvelo) - g_\temperature(t, \testadtemp),
\end{equation*}
where $\solution = (\testadvelo, \testadpres, \testadtemp)$ and $\adsolution = (\advelo, \adpres, \adtemp)$. Using the definitions of the involved mappings (cf. \eqref{eq:def_bilinear_forms}, \eqref{eq:def_linear_forms_state}, and \eqref{eq:averaged_linear_forms}), we see that \eqref{eq:sturm_averaged_adjoint} is equivalent to the following problem
\begin{equation}
	\label{eq:abstract_averaged_adjoint}
	\text{Find } \adsolution^t \in \statespace(\Omega) \text{ such that } \qquad e^*(t, \testadsolution, \adsolution^t) = 0 \qquad \text{ for all } \testadsolution \in \statespace(\Omega).
\end{equation}

To prove the well-posedness of problem \eqref{eq:abstract_averaged_adjoint}, we proceed similarly to Section~\ref{sec:analysis_state_system} and first investigate the boundedness of the right-hand sides.
\begin{Lemma}
	\label{lem:averaged_linear_forms}
	Let $\tau > 0$ be sufficiently small, $t\in [0,\tau]$, and $\adtemp \in \tspace(\Omega)$. Then, for all $t\in [0,\tau]$ the mappings $g_\temperature(t, \cdot)$ and $g_\velocity(t, \adtemp, \cdot)$ are linear and continuous. In particular, there exists a constant $C(\Omega) > 0$ independent of $t$ such that
	\begin{equation*}
		\begin{aligned}
			&\abs{g_\temperature(t, \testadtemp)} \leq C(\Omega) \norm{\testadtemp}{\tspace(\Omega)} \quad \text{ for all } \testadtemp \in \tspace(\Omega), \\
			&\abs{g_\velocity(t, \adtemp, \testadvelo)} \leq C(\Omega) \left( 1 + \norm{\adtemp}{\tspace(\Omega )} \right) \norm{\testadvelo}{\fspace(\Omega)} \quad \text{ for all } \adtemp\in \tspace(\Omega) \text{ and all } \testadvelo \in \fspace(\Omega).
		\end{aligned}
	\end{equation*}
\end{Lemma}
\begin{proof}
	The linearity of the mappings is readily seen. Their continuity follows from an application of H\"older's inequality and Lemma~\ref{lem:shape_mappings} analogously to the proof of Lemma~\ref{lem:operator_analysis}. The fact that the constants are independent of $t$ follows from the energy estimates of Lemma~\ref{lem:analysis_state_system} (cf. \eqref{eq:energy_estimates}).
\end{proof}

Now, we have the following result concerning the solvability of \eqref{eq:abstract_averaged_adjoint}.
\begin{Lemma}
	\label{lem:analysis_averaged_adjoint}
	Let $\tau > 0$ be sufficiently small. Then, for every $t \in [0, \tau]$ the averaged adjoint system \eqref{eq:abstract_averaged_adjoint} has a unique solution $\adsolution^t = (\advelo^t, \adpres^t, \adtemp^t) \in \statespace(\Omega)$ and there exists a constant $C(\Omega) > 0$ which is independent of $t$ such that
	\begin{equation*}
		\norm{\adsolution^t}{\statespace(\Omega)} = \norm{\advelo^t}{\fspace(\Omega)} + \norm{\adpres^t}{\pspace(\Omega)} + \norm{\adtemp^t}{\pspace(\Omega)} \leq C(\Omega) \quad \text{ for all } t\in [0,\tau].
	\end{equation*}
\end{Lemma}
\begin{proof}
	As can be seen from the definition of $e^*$, the averaged adjoint system is structurally very similar to the state system (cf. \eqref{eq:abstract_weak_state}). However, the order of the one-way coupling in the system is now reversed. In particular, problem \eqref{eq:abstract_averaged_adjoint} is equivalent to the following one-way coupled problem
	\begin{equation}
		\label{eq:one_way_adjoint_cd}
		\text{Find } \adtemp^t \in \tspace(\Omega) \text{ such that } \qquad c\left( t, \nicefrac{1}{2} (\velocity^t + \velocity^0), \testadtemp, \adtemp^t \right) = g_\temperature(t, \testadtemp) \qquad \text{ for all } \testadtemp \in \tspace(\Omega),
	\end{equation}
	as well as
	\begin{equation}
		\label{eq:one_way_adjoint_stokes}
		\left\lbrace\quad
		\begin{aligned}
			&\text{Find } (\advelo^t, \adpres^t) \in \fspace(\Omega) \times \pspace(\Omega) \text{ such that } \\
			&\qquad
			\begin{alignedat}{2}
				a(t, \testadvelo, \advelo^t) + b(t, \testadvelo, \adpres^t) &= g_\velocity(t, \adtemp^t, \testadvelo) \quad &&\text{ for } \adtemp^t \text{ solving } \eqref{eq:one_way_adjoint_cd} \text{ and all } \testadvelo \in \fspace(\Omega), \\
				b(t, \advelo^t, \testadpres) &= 0 \quad &&\text{ for all } \testadpres \in \pspace(\Omega).
			\end{alignedat}
		\end{aligned}
		\right.
	\end{equation}
	
	We want to apply the Lax-Milgram lemma to the adjoint convection-diffusion equation \eqref{eq:one_way_adjoint_cd}. Thanks to the proof of Lemma~\ref{lem:analysis_state_system}, we know from \eqref{eq:help_const_velo} that 
	\begin{equation*}
		\norm{\velocity^t + \velocity^0}{\fspace(\Omega)} \leq C(\Omega) \norm{\velocity\subin}{H^1(\Omega)^d}.
	\end{equation*}
	Thanks to Assumption~\ref{ass:smallness} we can use the results from Lemma~\ref{lem:operator_analysis} for $w = \nicefrac{1}{2}(\velocity^t + \velocity^0)$ and observe that $c\left(t, \nicefrac{1}{2}(\velocity^t + \velocity^0), \cdot, \cdot \right)$ is bilinear, continuous, and coercive.
As the right-hand side of \eqref{eq:one_way_adjoint_cd} is a continuous linear functional (cf. Lemma~\ref{lem:averaged_linear_forms}), we can apply the Lax-Milgram lemma to \eqref{eq:one_way_adjoint_cd} and get a unique solution $\adtemp^t \in \tspace(\Omega)$ together with a constant $C(\Omega) > 0$ such that
	\begin{equation}
		\label{eq:energy_adtemp}
		\norm{\adtemp^t}{\tspace(\Omega)} \leq C(\Omega) \quad \text{ for all } t\in [0,\tau].
	\end{equation}
	Note, that the constant is independent of $t$ due to Lemmas~\ref{lem:analysis_state_system} and~\ref{lem:averaged_linear_forms}.
	
	Thanks to Lemmas~\ref{lem:operator_analysis} and~\ref{lem:averaged_linear_forms} we apply, as in the proof of Lemma~\ref{lem:analysis_state_system}, the Lions-Lax-Milgram theorem to the adjoint Stokes system \eqref{eq:one_way_adjoint_stokes}. This yields a unique solution $(\advelo^t, \adpres^t) \in \fspace(\Omega) \times \pspace(\Omega)$ as well as a constant $C(\Omega) > 0$ such that
	\begin{equation*}
		\norm{\advelo^t}{\fspace(\Omega)} + \norm{\adpres^t}{\pspace(\Omega)} \leq C(\Omega) \quad \text{ for all } t\in [0,\tau].
	\end{equation*}
	Again, the constant is independent of $t$ due to Lemmas~\ref{lem:analysis_state_system} and~\ref{lem:averaged_linear_forms} and \eqref{eq:energy_adtemp}, completing the proof. \qedhere

\end{proof}

\begin{Remark}
	For $t=0$ we obtain the usual adjoint system
	\begin{equation}
		\label{eq:weak_adjoint}
		\left\lbrace\quad 
		\begin{aligned}
			&\text{Find } \adsolution^0 = (\advelo^0, \adpres^0, \adtemp^0) \in \statespace(\Omega) \text{ such that } \\
			&\qquad \integral{\Omega} \conductivity \grad \testadtemp \cdot \grad \adtemp^0 + \density \hcapacity \left(\velocity^0 + \velocity\subin\right) \cdot \grad \testadtemp\ \adtemp^0 \dx{x} + \integral{\Gamma\subwall} \htc\ \testadtemp \adtemp^0 \dx{s} \\
			&\qquad\ \ + \integral{\Omega} \viscosity\ D \testadvelo : D \advelo^0 - \testadpres\ \divergence{\advelo^0} - \adpres^0\ \divergence{\testadvelo} + \density \hcapacity\ \testadvelo \cdot \grad (\temperature^0 + \temperature\subin) \adtemp^0 \dx{x} \\
			&\quad =  2 \weighttemp \left( \flux(0, \temperature^0) - \fluxdes \right) \integral{\Gamma\subwall} \htc\ \testadtemp \dx{s} -2\weightvelo \integral{\Omega\subdomain} \left( (\velocity^0 + \velocity\subin) - \velocity\subdes \right) \cdot \testadvelo \dx{x}\\
			&\text{for all } \testadsolution = (\testadvelo, \testadpres, \testadtemp) \in \statespace(\Omega),
		\end{aligned}
		\right.
	\end{equation}
	where $\solution^0 = (\velocity^0, \pressure^0, \temperature^0)$ is the solution of the state system \eqref{eq:weak_state_reference}.
\end{Remark}
Hence, we have verified Assumption~\ref{ass:H2}.

\subsection*{Verification of Assumption~\ref{ass:H3}}

As noted in \cite{sturm}, we can verify Assumption~\ref{ass:H3} by proving the weak convergence of $\adsolution^t \rightharpoonup \adsolution^0$ in $\statespace(\Omega)$ and the weak continuity of the mapping $(t, \adsolution) \mapsto \partial_t \shapelagrangian(t, \solution^0, \adsolution)$. For this, we need the following result, whose proof is similar to the ones given in, e.g., \cite{lindemann, hohmann}.
\begin{Lemma}
	\label{lem:shape_continuity}
	Let $\solution^t \in \statespace(\Omega)$ and $\solution^0 \in \statespace(\Omega)$ be the solution of $\eqref{eq:abstract_weak_state}$ for $t\in [0,\tau]$ and $t=0$, respectively. Then, we have that
	\begin{equation}
		\label{eq:convergence_state}
		\lim\limits_{t \searrow 0} \norm{\solution^t - \solution^0}{\statespace(\Omega)} = 0.
	\end{equation}
\end{Lemma}
\begin{proof}
	Let $\solution^t = (\velocity^t, \pressure^t, \temperature^t)$ and $\solution^0 = (\velocity^0, \pressure^0, \temperature^0)$ be as above. Furthermore, we write $e_\velocity(t) := \velocity^t - \velocity^0$, $e_\pressure(t) := \pressure^t - \pressure^0$, and $e_\temperature(t) := \temperature^t - \temperature^0$. It is easy to see that \eqref{eq:convergence_state} is equivalent to
	\begin{equation*}
		\lim\limits_{t \searrow 0} \norm{e_\velocity(t)}{\fspace(\Omega)} = 0, \quad \lim\limits_{t \searrow 0} \norm{e_\pressure(t)}{\pspace(\Omega)} = 0, \quad \text{ and } \quad \lim\limits_{t \searrow 0} \norm{e_\temperature(t)}{\tspace(\Omega)} = 0,
	\end{equation*}
	which we show in the following. 
	Note, that the energy estimates \eqref{eq:energy_estimates} from Lemma~\ref{lem:analysis_state_system} combined with the triangle inequality yield
	\begin{equation}
		\label{eq:triangle}
		\norm{e_\velocity}{\fspace(\Omega)} \leq \norm{\velocity^t}{\fspace(\Omega)} + \norm{\velocity^0}{\fspace(\Omega)} \leq C(\Omega) \quad \text{ and } \quad \norm{e_\pressure}{\pspace(\Omega)} \leq \norm{\pressure^t}{\pspace(\Omega)} + \norm{\pressure^0}{\pspace(\Omega)} \leq C(\Omega).
	\end{equation}
	
	\subsubsection*{Convergence of the Velocity}
	
	Subtracting the weak forms \eqref{eq:one_way_stokes} for $t \in [0, \tau]$ and $t=0$ gives, after rearranging the resulting expression, 
	\begin{equation}
		\label{eq:test_velo}
		\begin{aligned}
			\integral{\Omega} \viscosity\ D e_\velocity(t) : D\testvelo \dx{x} = &\integral{\Omega} \viscosity \left( D(\velocity^t + \velocity\subin) (I - A(t)) \right) : D\testvelo \dx{x} - \integral{\Omega} \pressure^t\ \tr\left( D\testvelo (I - B(t)\transposed) \right) \dx{x}  \\
			& + \integral{\Omega} e_\pressure(t)\ \divergence{\testvelo} \dx{x} \quad \text{ for all } \testvelo \in \fspace(\Omega),
		\end{aligned}
	\end{equation}
	as well as
	\begin{equation}
		\label{eq:test_pres}
		\integral{\Omega} \testpres\ \divergence{e_\velocity(t)} \dx{x} = \integral{\Omega} \testpres\ \tr\left( D(\velocity^t + \velocity\subin) (I - B(t)\transposed) \right) \dx{x} \quad \text{ for all } \testpres \in \pspace(\Omega).
	\end{equation}
	Due to H\"older's inequality and \eqref{eq:energy_estimates} we have for all $\testvelo \in \fspace(\Omega)$ that
	\begin{align}
		\label{eq:holder_velo1}
		&\abs{\integral{\Omega} \viscosity\ D e_\velocity(t) : D\testvelo \dx{x}} \leq C \norm{e_\velocity(t)}{\fspace(\Omega)} \norm{\testvelo}{\fspace(\Omega)}, \\
		\label{eq:holder_velo2}
		&\abs{\integral{\Omega} \viscosity \left( D(\velocity^t + \velocity\subin) (I - A(t)) \right) : D\testvelo \dx{x}} \leq C(\Omega) \norm{I - A(t)}{L^\infty(\Omega)^{d\times d}} \norm{\testvelo}{\fspace(\Omega)}, \\
		\label{eq:holder_velo3}
		&\abs{\integral{\Omega} \pressure^t\ \tr\left( D\testvelo (I - B(t)\transposed) \right) \dx{x}} \leq C(\Omega) \norm{I - B(t)\transposed}{L^\infty(\Omega)^{d\times d}} \norm{\testvelo}{\fspace(\Omega)}.
	\end{align}
	Further, Poincar\'e's inequality gives
	\begin{equation}
		\label{eq:left_velo}
		C(\Omega) \norm{e_\velocity(t)}{\fspace(\Omega)}^2 \leq \integral{\Omega} \viscosity\ D e_\velocity(t) : D e_\velocity(t) \dx{x} \quad \text{ for all } t\in [0,\tau].
	\end{equation}
	As $\pspace(\Omega) = L^2(\Omega)$, the Riesz representation theorem and \eqref{eq:test_pres} reveals that
	\begin{equation*}
		\divergence{e_\velocity(t)} = \tr\left( D\left( \velocity^t + \velocity\subin \right) \left( I - B(t)\transposed \right) \right) \quad \text{ in } \pspace(\Omega),
	\end{equation*}
	 and H\"older's inequality in addition to \eqref{eq:energy_estimates} then gives
	\begin{equation*}
		\norm{\divergence{e_\velocity(t)}}{\pspace(\Omega)} \leq C(\Omega) \norm{I - B(t)\transposed}{L^\infty(\Omega)^{d\times d}}.
	\end{equation*}
	Using the previous estimate in addition to H\"older's inequality and \eqref{eq:triangle} reveals that
	\begin{equation}
		\label{eq:rhs2_velo}
		\abs{\integral{\Omega} e_\pressure(t)\ \divergence{e_\velocity(t)} \dx{x}} \leq \norm{e_\pressure(t)}{\pspace(\Omega)} \norm{\divergence{e_\velocity(t)}}{\pspace(\Omega)} \leq C(\Omega) \norm{I - B(t)\transposed}{L^\infty(\Omega)^{d\times d}}.
	\end{equation}
	Choosing $\testvelo = e_\velocity(t)$ in \eqref{eq:test_velo}, estimating the term on the left of this equation by \eqref{eq:left_velo} and those on the other side by \eqref{eq:holder_velo2}, \eqref{eq:holder_velo3}, \eqref{eq:rhs2_velo}, and using \eqref{eq:triangle} to estimate $\norm{e_\velocity(t)}{\fspace(\Omega)}$ on the right-hand side gives
	\begin{equation*}
		\norm{e_\velocity(t)}{\fspace(\Omega)}^2 \leq C(\Omega) \left( \norm{I - A(t)}{L^\infty(\Omega)^{d\times d}} + \norm{I - B(t)\transposed}{L^\infty(\Omega)^{d\times d}} \right). 
	\end{equation*}
	As $C(\Omega)$ is independent of $t$, we get from the above estimate and Lemma~\ref{lem:shape_mappings} that
	\begin{equation}
		\label{eq:convergence_velo}
		\lim\limits_{t \searrow 0} \norm{e_\velocity(t)}{\fspace(\Omega)} = 0.
	\end{equation}
	
	\subsubsection*{Convergence of the Pressure}
	
	Isolating the final term on the right-hand side of \eqref{eq:test_velo} and estimating the remaining terms by \eqref{eq:holder_velo1}, \eqref{eq:holder_velo2}, and \eqref{eq:holder_velo3} gives
	\begin{equation}
		\label{eq:help_pressure}
		\integral{\Omega} e_\pressure(t)\ \divergence{\testvelo} \dx{x} \leq C(\Omega) \norm{\testvelo}{\fspace(\Omega)} \left( \norm{I - A(t)}{L^\infty(\Omega)^{d\times d}} + \norm{I - B(t)\transposed}{L^\infty(\Omega)^{d\times d}} + \norm{e_\velocity(t)}{\fspace(\Omega)} \right).
	\end{equation}
	Since the divergence operator $\text{div} \colon \fspace(\Omega) \to \pspace(\Omega)$ is surjective (cf. Lemma~\ref{lem:operator_analysis}), there exists a $\testvelo(t) \in  \fspace(\Omega)$ such that $\divergence{\testvelo(t)} = e_\pressure(t)$ in $\pspace(\Omega)$ for all $t\in [0,\tau]$. Moreover, as a consequence of the open mapping theorem (see, e.g., \cite{alt} or \cite[Lemma~A.36]{ern_guermond}) there exists a constant $C(\Omega) > 0$ such that
	\begin{equation}
		\label{eq:open_mapping}
		\norm{\testvelo(t)}{\fspace(\Omega)} \leq C(\Omega) \norm{e_\pressure(t)}{\pspace(\Omega)} \leq C(\Omega),
	\end{equation}
	where we used \eqref{eq:triangle} for the last estimate. Choosing $\testvelo = \testvelo(t)$ in \eqref{eq:help_pressure} and using \eqref{eq:open_mapping} yields
	\begin{equation*}
		\begin{aligned}
			&\norm{e_\pressure(t)}{\pspace(\Omega)}^2 = \integral{\Omega} e_\pressure(t)\ \divergence{\testvelo(t)} \dx{x} \\
			\leq\ &C(\Omega) \left( \norm{I - A(t)}{L^\infty(\Omega)^{d\times d}} + \norm{I - B(t)\transposed}{L^\infty(\Omega)^{d\times d}} + \norm{e_\velocity(t)}{\fspace(\Omega)} \right).
		\end{aligned}
	\end{equation*}
	Due to Lemma~\ref{lem:shape_mappings} and \eqref{eq:convergence_velo}, we get from the previous estimate that $\lim_{t \searrow 0} \norm{e_\pressure(t)}{\pspace(\Omega)} = 0$.
	
	\subsubsection*{Convergence of the Temperature}
	
	Finally, we investigate the convergence of the temperature component. Subtracting \eqref{eq:one_way_temperature} for $t\in [0, \tau]$ and $t=0$ yields, after rearranging and testing with $\testtemp = e_\temperature(t)$, the following
	\begin{equation}
		\label{eq:tested_temp}
		\begin{aligned}
			&\integral{\Omega} \conductivity \grad e_\temperature(t) \cdot \grad e_\temperature(t) \dx{x} + \integral{\Gamma\subwall} \htc\ e_\temperature(t)^2 \dx{x} + \integral{\Omega} \density\hcapacity \left(\velocity^0 + \velocity\subin\right)\cdot \grad e_\temperature(t)\ e_\temperature(t) \dx{x} \\
			= &\integral{\Omega} \conductivity \left( (I - A(t)) \grad (\temperature^t + \temperature\subin) \right) \cdot \grad e_\temperature(t) \dx{x} + \integral{\Omega} \density\hcapacity \left(\velocity^t + \velocity\subin\right) \cdot \left( (I - B(t)) \grad (\temperature^t + \temperature\subin)\right) e_\temperature(t) \dx{x} \\
			&- \integral{\Omega} \density\hcapacity\ e_\velocity(t) \cdot \grad (\temperature^t + \temperature\subin)\ e_\temperature(t) \dx{x} + \integral{\Gamma\subwall} \htc \left((\temperature^t + \temperature\subin) - \temperature\subwall \circ \flow\right) e_\temperature(t) (1 - \omega(t)) \dx{s} \\
			&+ \integral{\Gamma\subwall} \htc \left( \temperature\subwall \circ \flow - \temperature\subwall \right)e_\temperature(t) \dx{s}.
		\end{aligned}
	\end{equation}
	Using Poincar\'e's inequality and Assumption~\ref{ass:smallness} analogously to the proof of Lemma~\ref{lem:operator_analysis} reveals that
	\begin{equation}
		\label{eq:estimate_temp_poincare}
		C(\Omega) \norm{e_\temperature(t)}{\tspace(\Omega)}^2 \leq \integral{\Omega} \conductivity \grad e_\temperature(t) \cdot \grad e_\temperature(t) \dx{x} + \integral{\Gamma\subwall} \htc\ e_\temperature(t)^2 \dx{x} + \integral{\Omega} \density\hcapacity \left(\velocity^0 + \velocity\subin\right)\cdot \grad e_\temperature(t)\ e_\temperature(t) \dx{x}.
	\end{equation}
	For the terms of the right-hand side of \eqref{eq:tested_temp} we use H\"older's inequality, Lemma~\ref{lem:shape_composition} and \eqref{eq:energy_estimates} to obtain
	\begin{equation}
		\label{eq:estimate_rhs}
		\begin{aligned}
			&\abs{\integral{\Omega} \conductivity \left( (I - A(t)) \grad (\temperature^t + \temperature\subin) \right) \cdot \grad e_\temperature(t) \dx{x}} \leq C(\Omega) \norm{I - A(t)}{L^\infty(\Omega)^{d\times d}} \norm{e_\temperature(t)}{\tspace(\Omega)}, \\
			&\abs{\integral{\Omega} \density\hcapacity \left(\velocity^t + \velocity\subin\right) \cdot \left( (I - B(t)) \grad (\temperature^t + \temperature\subin)\right) e_\temperature(t) \dx{x}} \leq C(\Omega) \norm{I - B(t)}{L^\infty(\Omega)^{d\times d}} \norm{e_\temperature(t)}{\tspace(\Omega)}, \\
			&\abs{\integral{\Omega} \density\hcapacity\ e_\velocity(t) \cdot \grad (\temperature^t + \temperature\subin)\ e_\temperature(t) \dx{x}} \leq C(\Omega) \norm{e_\velocity(t)}{\fspace(\Omega)} \norm{e_\temperature(t)}{\tspace(\Omega)}, \\
			&\abs{\integral{\Gamma\subwall} \htc \left((\temperature^t + \temperature\subin) - \temperature\subwall \circ \flow\right) e_\temperature(t) (1 - \omega(t)) \dx{s}} \leq C(\Omega) \norm{1 - \omega(t)}{L^\infty(\Omega)} \norm{e_\temperature(t)}{\tspace(\Omega)}, \\
			&\abs{\integral{\Gamma\subwall} \htc \left( \temperature\subwall \circ \flow - \temperature\subwall \right)e_\temperature(t) \dx{s}} \leq C \norm{\temperature\subwall \circ \flow - \temperature\subwall}{H^1(\R^d)} \norm{e_\temperature(t)}{\tspace(\Omega)}.
		\end{aligned}
	\end{equation}
	Plugging estimates \eqref{eq:estimate_temp_poincare} and \eqref{eq:estimate_rhs} into \eqref{eq:tested_temp} and dividing by $\norm{e_\temperature(t)}{\tspace(\Omega)}$ gives
	\begin{equation*}
		\begin{aligned}
			\norm{e_\temperature(t)}{\tspace(\Omega)} \leq\ C(\Omega) &\left( \norm{I - A(t)}{L^\infty(\Omega)^{d\times d}} + \norm{I - B(t)}{L^\infty(\Omega)^{d\times d}} + \norm{e_\velocity(t)}{\fspace(\Omega)} \right. \\
			&\left. \quad  + \norm{1 - \omega(t)}{L^\infty(\Omega)} + \norm{\temperature\subwall \circ \flow - \temperature\subwall}{H^1(\R^d)} \right).
		\end{aligned}
	\end{equation*}
	Similarly to before, using \eqref{eq:convergence_velo} as well as Lemmas~\ref{lem:shape_mappings} and~\ref{lem:shape_composition} reveals that $\lim_{t\searrow 0}\norm{e_\temperature(t)}{\tspace(\Omega)} = 0$. Altogether, we have shown that $\lim_{t\searrow 0} \norm{\solution^t - \solution^0}{\statespace(\Omega)} = 0$, completing the proof.
\end{proof}

Using this result, we can now show the weak convergence of $\adsolution^t \rightharpoonup \adsolution^0$ in $\statespace(\Omega)$ as follows.

\begin{Lemma}
	Let $\set{\adsolution^t} = Y(t, \solution^t, \solution^0)$ be the solution set of the averaged adjoint system \eqref{eq:abstract_averaged_adjoint} and let $\set{\adsolution^0} = Y(0, \solution^0)$ be the solution set of the usual adjoint system \eqref{eq:weak_adjoint}. Then, we have for $t\to 0$ that
	\begin{equation*}
		\adsolution^t \rightharpoonup \adsolution^0 \text{ in } \statespace(\Omega).
	\end{equation*}
\end{Lemma}
\begin{proof}
	Let $(t_n)_{n\in \mathbb{N}} \subset \R$ be a sequence with $t_n \geq 0$ and $\lim_{n\to \infty} t_n = 0$. As a consequence of Lemma~\ref{lem:analysis_averaged_adjoint} we see that $\adsolution^{t_n}$ is uniformly bounded in $\statespace(\Omega)$, which is a Hilbert space. Therefore, there exists a subsequence $(t_{n_k})_{k\in \mathbb{N}}$ and an element $Q\in \statespace(\Omega)$ such that $\adsolution^{t_{n_k}} \rightharpoonup Q$ in $\statespace(\Omega)$. Due to Lemmas~\ref{lem:shape_mappings}, \ref{lem:shape_composition}, and~\ref{lem:shape_continuity}, and \cite[Remark~6.3]{alt} we observe that $e^*(t,\testadsolution, \adsolution^t) \to e^*(0, \testadsolution, Q)$ for all $\testadsolution \in \statespace(\Omega)$. In particular, we get $e^*(0,\testadsolution, Q) = 0$ for all $\testadsolution \in \statespace(\Omega)$. Therefore, we indeed have $Q = \adsolution^0$ due to the uniqueness of the limit (cf. Lemma~\ref{lem:analysis_averaged_adjoint}).
\end{proof}

Now, we prove the weak continuity of the mapping $(t, \adsolution) \mapsto \partial_t \shapelagrangian(t, \solution^0, \adsolution)$. 

\begin{Lemma}
	The mapping
	\begin{equation*}
		[0, \tau] \times \statespace(\Omega) \to \R; \quad (t, \adsolution) \mapsto \partial_t \shapelagrangian(t, \solution^0, \adsolution)
	\end{equation*}
	is weakly continuous.
\end{Lemma}
\begin{proof}
	Let $t^n \rightharpoonup t$ in $\R$ and $\adsolution^n \rightharpoonup \adsolution$ in $\statespace(\Omega)$. Thanks to Assumption~\ref{ass:H1} we can differentiate $\shapelagrangian$ w.r.t.\ $t$ and end up with the formula given in \eqref{eq:partial_time}. Due to Lemmas~\ref{lem:shape_mappings} and~\ref{lem:shape_composition} we get the strong convergence of all mappings involving $t$. Using \cite[Remark~6.3]{alt} we get that $\partial_t \shapelagrangian(t^n, \solution^0, \adsolution^n) \to \partial_t \shapelagrangian(t, \solution^0, \adsolution)$ in $\R$, proving the weak continuity as desired.
\end{proof}
This completes the verification of Assumption~\ref{ass:H3}. Hence, we have proved that the assumptions of Theorem~\ref{thm:sturm} are indeed satisfied for our problem and we obtain the shape differentiability of our cost functional. Finally, we summarize this result in the following theorem.
\begin{Theorem}
	\label{thm:shape_derivative}
	Let $\Omega \in \admissiblegeom$, $\vectorfield \in \admissibledefo$, and $\flow$ be the flow associated to $\vectorfield$. Then, the reduced cost functional \eqref{eq:reduced_cost_function} is shape differentiable at $\Omega$ w.r.t. $\admissibledefo$ with shape derivative
	\begin{equation*}
		\left\lbrace \quad
		\begin{aligned}
			&d\reducedcostfunction(\Omega)[\vectorfield] \\
			=\ &2\weighttemp \left( \flux(0, \temperature^0) - \fluxdes \right) \integral{\Gamma\subwall} \htc \left( \temperature\subwall - (\temperature^0 + \temperature\subin) \right) \tdiv{\vectorfield} + \htc\ \grad \temperature\subwall \cdot \vectorfield \dx{s} \\
			&+ \weightvelo \integral{\Omega\subdomain} \abs{(\velocity^0 + \velocity\subin) - \velocity\subdes}^2 \divergence{\vectorfield} - 2\left( (\velocity^0 + \velocity\subin) - \velocity\subdes \right) \cdot \left( D\velocity\subdes\ \vectorfield \right) \dx{x} \\
			&+ \weightreg \integral{\Gamma} \tdiv{\vectorfield} \dx{s} + \integral{\Omega} \viscosity \left(D(\velocity^0 + \velocity\subin) (\divergence{\vectorfield}I- 2\varepsilon(\vectorfield)) \right) : D\advelo^0 \dx{x} \\
			&- \integral{\Omega} \pressure^0\ \tr\left( D\advelo^0 (\divergence{\vectorfield}I - D\vectorfield ) \right) \dx{x} - \integral{\Omega} \adpres^0\ \tr\left( D(\velocity^0 + \velocity\subin) (\divergence{\vectorfield}I - D\vectorfield) \right) \dx{x} \\
			&+ \integral{\Omega} \conductivity \left( (\divergence{\vectorfield}I - 2\varepsilon(\vectorfield)) \grad (\temperature^0 + \temperature\subin ) \right) \cdot \grad \adtemp^0 \dx{x} \\
			&+ \integral{\Omega} \density \hcapacity \left(\velocity^0 + \velocity\subin\right) \cdot \left((\divergence{\vectorfield}I - D\vectorfield^\top) \grad (\temperature^0 + \temperature\subin) \right) \ \adtemp^0 \dx{x} \\
			&+ \integral{\Gamma\subwall} \htc \left((\temperature^0 + \temperature\subin) - \temperature\subwall\right) \adtemp^0\ \tdiv{\vectorfield} - \htc \grad \temperature\subwall \cdot \vectorfield \adtemp^0 \dx{s},
		\end{aligned}
		\right.
	\end{equation*}
	where $\solution^0 = (\velocity^0, \pressure^0, \temperature^0)$ solves the state system \eqref{eq:weak_state_reference} and $\adsolution^0 = (\advelo^0, \adpres^0, \adtemp^0)$ solves the adjoint system \eqref{eq:weak_adjoint}.
\end{Theorem}
\begin{proof}
	We apply Theorem~\ref{thm:sturm}, whose assumptions have been verified throughout this section, to the shape Lagrangian $\shapelagrangian$ defined in Section~\ref{sec:relation_to_abstract}. In combination with \eqref{eq:relation} we can then calculate the shape derivative in the following way
	\begin{equation*}
		d\reducedcostfunction(\Omega)[\vectorfield] = \lim\limits_{t\searrow 0} \frac{\reducedcostfunction(\Omega_t) - \reducedcostfunction(\Omega)}{t} = \left. \frac{d}{d t} \reducedcostfunction(\Omega_t) \right\rvert_{t=0^+} = \left. \frac{d}{d t} \shapelagrangian(t, \solution^t, \adsolution) \right\rvert_{t=0^+} = \partial_t \shapelagrangian(t, \solution^0, \adsolution^0).
	\end{equation*}
	Due to Lemmas~\ref{lem:shape_mappings} and~\ref{lem:shape_composition} as well as the formula \eqref{eq:partial_time} it is easy to see that we indeed have the claimed expression for the shape derivative. Furthermore, $d\reducedcostfunction(\Omega)[\vectorfield]$ is linear in $\vectorfield$ and continuous in the $C^\infty_0(\R^d; \R^d)$ topology. Therefore, it is in fact the shape derivative of the reduced cost functional.
\end{proof}

This completes the theoretical analysis of the shape optimization problem from Section~\ref{sec:problem_formulation}. 
\begin{Remark*}
	Our results are easily transferred to the other flow models which we presented in \cite{blauth}. This is due to the fact that these models differ from the one here only in additional linear terms and use state spaces adapted to each model. However, the former only enhance the coercivity of the corresponding bilinear forms and the latter are structurally very similar to $\statespace(\Omega)$. Therefore, the proofs can be carried out analogously, and we obtain indeed the shape differentiability of all models considered \cite{blauth}. We remark that this shape differentiability result only leads to a local improvement of the geometry, and the investigation of the existence of an optimal shape for problem \eqref{eq:opt_problem} is beyond the scope of this paper.
	
	Note, that we consider only the volume formulation of the shape derivative, as given in Theorem~\ref{thm:shape_derivative}, due to the following. To derive a boundary or Hadamard formulation of the shape derivative in the spirit of the structure theorem (cf. \cite[Chapter 9, Theorem 3.6]{delfour_zolesio}), we would need additional smoothness assumptions. The boundary formulation only exists for domains $\Omega$ whose boundary is of class $C^k$ with $k\geq 1$ (cf. \cite[Chapter 9, Theorem 3.6 and Corollary 1]{delfour_zolesio}), whereas we only consider Lipschitz domains in this paper, as motivated by the application. In particular, the geometries depicted in Figures~\ref{figure:full_geometry} and~\ref{figure:optimized} only have a Lipschitz boundary, and are not sufficiently smooth to satisfy the assumptions of the structure theorem. Moreover, for the numerical solution of shape optimization problems it is discussed in \cite{paganini_comparison} that the volume formulation has superior approximation properties. Hence, for the numerical optimization of the physical application the volume formulation is to be preferred over the boundary formulation. For these reasons, a derivation of the boundary formulation is not feasible within the setting of this paper. 
\end{Remark*}

\section{Conclusion and Outlook}
\label{sec:conclusion}

In this paper, we investigated the theoretical framework of a shape optimization problem for a microchannel cooling system \cite{blauth}. The presented  mathematical model for the  microchannel cooling system consists of a Stokes system coupled to a convection-diffusion equation. Furthermore, we presented a generalization of the shape optimization problem we studied numerically in \cite{blauth}. In particular, we use a cost functional based on the tracking of the energy absorbed by the cooler as well as some desired fluid velocity on a subdomain of the system. After recalling some basic results from shape optimization, we analyzed the state system on a domain transformed by the speed method. In particular, we have shown the well-posedness of the transformed state system for small transformations. Furthermore, we proved that the cost functional of our optimization problem is shape differentiable and calculated its shape derivative by an application of the material derivative free adjoint approach of \cite{sturm}.

For future work we propose to examine the existence of a minimizer of the optimization problem \eqref{eq:opt_problem}, which is a non-trivial task. Additionally, one could calculate a shape Hessian of the cost functional which could then be used in a shape Newton method to solve the optimization problem. 
Finally, a topological sensitivity analysis could be also used to further optimize the geometry of the cooler.

\section*{Acknowledgments}

S. Blauth gratefully acknowledges financial support from the Fraunhofer Institute for Industrial Mathematics (ITWM).

\bibliographystyle{siam}
\bibliography{lit}

\end{document}